\documentclass[english,11pt]{smfart}

\usepackage{etex}

\usepackage{amsbsy}
\usepackage{amsmath,amsfonts,amssymb,amsthm,mathrsfs,mathtools}

\usepackage{calrsfs}

\usepackage{bm}

\usepackage[a4paper,vmargin={3cm,3cm},hmargin={3.5cm,3.5cm}]{geometry}
\usepackage[font=sf, labelfont={sf,bf}, margin=1cm]{caption}
\usepackage{graphicx}
\usepackage{epsfig}
\usepackage{latexsym}
\usepackage{xcolor}
\usepackage{ae,aecompl}
\usepackage{soul,framed}
\usepackage{comment}

\usepackage{xcolor}
\usepackage[pdfpagemode=UseNone,bookmarksopen=false,colorlinks=true,urlcolor=blue,citecolor=blue,citebordercolor=blue,linkcolor=blue]{hyperref}
\usepackage{smfhyperref}
\usepackage[capitalize]{cleveref}

\usepackage{pstricks}
\usepackage{enumerate}
\usepackage{tikz,animate,media9}						
\usepackage{todonotes}
\usepackage{pifont}
\usepackage{bm,marvosym}
\usepackage{algorithm}
\usepackage{algorithmic}

\usepackage{bbm}

\usepackage{tcolorbox}

\definecolor{aleacolor}{rgb}{0.16,0.59,0.78}

\hypersetup{
	breaklinks,
	colorlinks=true,
	linkcolor=aleacolor,
	urlcolor=aleacolor,
	citecolor=aleacolor}

\newcount\colveccount
\newcommand*\colvec[1]{
	\global\colveccount#1
	\begin{pmatrix}
		\colvecnext
	}
	\def\colvecnext#1{
		#1
		\global\advance\colveccount-1
		\ifnum\colveccount>0
		\\
		\expandafter\colvecnext
		\else
	\end{pmatrix}
	\fi
}

\newcommand{\ndN}{\mathbb{N}}

\newcommand{\ndR}{\mathbb{R}}

\renewcommand{\Pr}[1]{\mathbb{P}(#1)}

\newcommand{\Prb}[1]{\mathbb{P}\left(#1\right)}

\newcommand{\Ex}[1]{\mathbb{E}[#1]}

\newcommand{\Exb}[1]{\mathbb{E}\left[#1\right]}

\newcommand{\Va}[1]{\mathbb{V}[#1]}

\newcommand{\one}{{\mathbbm{1}}}

\newcommand{\convdis}{\,{\buildrel \mathrm{d} \over \longrightarrow}\,}

\newcommand{\convp}{\,{\buildrel \mathrm{p} \over \longrightarrow}\,}

\newcommand{\eqdist}{\,{\buildrel \mathrm{d} \over =}\,}

\newcommand{\He}{\mathrm{H}}

\newcommand{\Di}{\mathrm{D}}

\newcommand{\cE}{\mathcal{E}}
\newcommand{\cF}{\mathcal{F}}

\newcommand{\cS}{\mathcal{S}}
\newcommand{\cT}{\mathcal{T}}

\newcommand{\mA}{\mathsf{A}}

\newcommand{\mF}{\mathsf{F}}

\newcommand{\mK}{\mathsf{K}}
\newcommand{\mL}{\mathsf{L}}

\newcommand{\mT}{\mathsf{T}}

\newtheorem{theorem}{Theorem}[section]

\newtheorem{corollary}[theorem]{Corollary}
\newtheorem{proposition}[theorem]{Proposition}
\newtheorem{lemma}[theorem]{Lemma}

\numberwithin{equation}{section}


\title{\textbf{Poisson--Dirichlet scaling limits of Kemp's  supertrees}}
\date{}

\author{Benedikt Stufler}

\address[Benedikt Stufler]{Vienna University of Technology}
\email{benedikt.stufler at tuwien.ac.at}

\begin{document}

\vspace {-0.5cm}

\begin{abstract}
	We determine the Gromov--Hausdorff--Prokhorov scaling limits and local limits of Kemp's $d$-dimensional binary trees and other models of supertrees. The limits exhibit a root vertex with infinite degree and are constructed by rescaling infinitely many independent stable trees or other spaces according to a function of a two-parameter Poisson--Dirichlet process and gluing them together at their roots.  We discuss universality aspects of random spaces constructed in this fashion and sketch a phase diagram.
\end{abstract}


\maketitle

\section{Introduction and main result}

Trees are fundamental  objects in computer science. A great number of algorithms rely on some variants of these structures~\cite{MR2251473,MR2483235}. Studying the typical shape of trees hence can yield insights on the typical performance of such algorithms~\cite{MR2483235}.

Apart from the average case analysis of algorithms, the construction of continuous limit objects of growing sequences of random trees has also lead to flowering research field at the intersection of combinatorics and probability theory~\cite{MR1085326,MR2484382}.

Numerous variants and classes of trees exist. In the present work, we are concerned with so-called supertrees. These structures are ``trees of trees'', meaning we start with a tree from some class, and on each vertex of that tree we graft another tree from a second class. The vertices of the first tree may be referred to as the first level of the supertree, and the remaining vertices form the second level. For this reason, we say such a supertree is $2$-dimensional.

Of course, this construction may be iterated, if the trees grafted upon the vertices of the first level are also supertrees. Hence, given $d \ge 2$ classes of  trees, we may form the corresponding $d$-dimensional (super-)trees.

Motivated by applications to computer science, Kemp~\cite{zbMATH00168426} constructed  and studied such a  class $\cF_d$ of $d$-dimensional binary trees, defined as follows. We start with the class $\cF_1$ of rooted trees where each vertex has either no child, or a left child, or a right child, or both a left and a right child. For $d \ge 2$, the class $\cF_d$ consists of such a tree, where for each vertex $v$ we add an edge between $v$ and the root of some tree from~$\cF_{d-1}$.  

Kemp's $d$-dimensional binary supertrees are  linked via combinatorial bijections to monotonically labelled trees~\cite{MR1260502}, which have received further attention in~\cite{MR2193117}. Core sizes in supertrees were studied recently in the general context of composition schemes~\cite{banderier2021analyticv2}.

The generating series of $\cF_d$ is defined by $F_d(x) = \sum_{n \ge 1} f_{d,n} x^n$, with $f_{n,d}$ denoting the number of $n$-vertex trees in the class $\cF_d$. It is given by the recursive list of equations
\begin{align*}
	F_1(x) &= x(1 + 2F_1(x) + F_1(x)^2), \\
	F_d(x) &= F_1(x F_{d-1}(x)), \qquad d \ge 2.
\end{align*}
Kemp~\cite[Thm. 6]{zbMATH00168426} elegantly proved the asymptotic enumerative formula
\begin{align}
	\label{eq:fdasymp}
	f_{d,n} \sim \frac{2^{2(1- 2^{-d})}}{|\Gamma(-2^{-d})|} 2^{2n} n^{-1 -2^{-d}}
\end{align}
as $d$ is fixed and $n$ tends to infinity.
This contrasts typical combinatorial classes of rooted trees which exhibit the polynomial factor $n^{-3/2}$ instead. Large trees from such classes have been shown to admit Aldous' Brownian tree~\cite{MR1085326,MR1166406,MR1207226}  as limit \cite{MR3050512,MR3773800,MR3983790,MR3573447,MR3914354}, after rescaling distances by $n^{-1/2}$ times some scaling constant that depends on the class under consideration. This phenomenon is referred to as universality of the Brownian tree.

The polynomial factor $n^{-1-2^{-d}}$ suggests that for $d \ge 2$ the asymptotic shape of a uniformly drawn $n$-vertex tree $\mF_{d,n}$ from the subset $\cF_{d,n}$ of $n$-vertex trees in $\cF_d$ may be quite different. This  motivates the question if they have a scaling limit and if it differs from the Brownian tree.

We answer this question, showing that $\mF_{d,n}$ has an inhomogeneous scaling limit $\cS_d$ whose root vertex has countably infinite degree. 
The Brownian tree, on the other hand,  has the property, that almost surely all of its vertices either have degree three or one~\cite{MR1085326}, hence we obtain a different limit.  Our  results also go beyond the specific class of $d$-dimensional binary supertrees, as we construct the limit $\cS_d$ as a special case of a multiparameter family of random spaces that appear to be universal limiting objects for random supertrees and other super-structures.

In order to state this result,  let $0<\alpha<1$ and $\theta > - \alpha$ be given parameters and let $V_1>V_2>\ldots$ denote the ordered points of the two-parameter Poisson--Dirichlet distribution $\mathrm{PD}(\alpha, \theta)$ introduced by Pitman and Yor~\cite{MR1434129}. Let $L$ denote the law of a random rooted compact metric space equipped with a Borel probability measure. Let $X_1, X_2, \ldots$ denote independent samples of $L$ (which are also independent from $(V_i)_{i \ge 1}$). Given a deterministic parameter $s>0$, we let $\cS(\alpha, \theta, s, L)$ denote the random rooted measured real tree obtained by gluing the rescaled spaces $V_1^s X_1, V_2^s X_2, \ldots$ together at their roots. Here, \emph{rescaling} means multiplying distances and measures by the corresponding factor, and \emph{gluing} means identifying the root-points. A formal description of this construction is detailed in Section~\ref{sec:glue}.

Let $d_{\mF_{d,n}}$ denote the graph distance on $\mF_{d,n}$, and let $\mu_{\mF_{d,n}}$ denote the uniform measure on the vertex set of $\mF_{d,n}$. We view $\mF_{d,n}$ as marked at its root vertex. Let $L_{\mathrm{Brownian}}$ denote the law of Aldous' Brownian tree, constructed from Brownian excursion of duration one. This way, by the pioneering invariance principle for large simply generated trees~\cite{MR1207226}, the Brownian tree is the scaling limit of $\mF_{1,n}$ after rescaling distances by $2^{-3/2} n^{-1/2}$. Our main result describes the limit for $d \ge 2$:

\begin{theorem}
	\label{te:main}
	For each integer $d \ge 2$, 
	\[
		(\mF_{d,n}, 2^{-3/2} n^{-1/2} d_{\mF_{d,n}}, \mu_{\mF_{d,n}}) \convdis  \cS(1/2, -1/2^d, 1/2, L_{\mathrm{Brownian}})
	\]
	as $n \to \infty$ with respect to the rooted Gromov--Hausdorff--Prokhorov metric.
	
\end{theorem}

The proof recovers Equation~\eqref{eq:fdasymp} by a probabilistic argument. We complete this proof in Section~\ref{sec:mainproof}, where we also verify the following tail-bounds for the height $\He(\mF_{d,n})$.
\begin{theorem}
	\label{te:bound}
	For each integer $d \ge 1$, there exist constants $C,c>0$ such that for all $n$ and all $x>0$
	\[
		\Pr{\He(\mF_{d,n}) \ge x} \le C \exp(-cx^2/n).
	\]
\end{theorem}
For $d=1$ such bounds have already been established in~\cite{MR3077536} (see  also~\cite{addarioberry2022random}), and our proof builds upon the results there. Theorem~\ref{te:bound} ensures arbitrarily high uniform integrability of the rescaled height $\He(\mF_{d,n}) / \sqrt{n}$. Together with Theorem~\ref{te:main} this implies that for any $p \ge 1$ we have that 
\[
	\Ex{\He(\mF_{d,n})^p} 2^{-3p/2} n^{-p/2} \to \Ex{\He( \cS(1/2, -1/2^d, 1/2, L_{\mathrm{Brownian}}))^p} < \infty,
\]
and likewise for similar parameters such as the diameter, or the distance between two random points in $\mF_{d,n}$.

In Section~\ref{sec:phase} we  additionally show that spaces of the form $\cS(\alpha, \theta, s, L)$ have a universality property, meaning that they appear as scaling limits of many other classes of supertrees~\cite{MR1370950} and super-structures. However, of course, not all supertrees admit such limits, other phases exist as well. We sketch a phase diagram and present several examples. In particular, Section~\ref{sec:phase} presents several additional results, which we do not state here in the introduction, as each requires a substantial amount of notation and background to be recalled. Our main tool for establishing these results are limits for Gibbs partitions~\cite{gibbsmcta}.

Let us mention that Aldous and Pitman~\cite{MR1808372} introduced inhomogeneous continuum random trees which can have a finite or infinite number of points called \emph{hubs} with infinite degree. It is not clear to the author of the present work whether there is a connection between these trees and those constructed from the two-parameter Poisson--Dirichlet process here.

Apart from studying the global shape of Kemp's supertrees we additionally obtain local convergence:
\begin{theorem}
	\label{te:main2}
	There is a random infinite tree $\hat{\mF}_d$ that depends on $d$ such that
	\[
		\mF_{d,n} \convdis \hat{\mF}_d
	\]
	in the local topology. Furthermore, if $v_n$ denotes a uniformly at random selected vertex of $\mF_{d,n}$, then there is a random tree $\bar{\mF}$ that does not depend on $d$ such that
	\[
		(\mF_{d,n},v_n) \convdis \bar{\mF}.
	\]
	The convergence also holds in the quenched sense, so that
	\[
		\mathfrak{L}((\mF_{d,n},v_n) \mid \mF_{d,n}) \convp \mathfrak{L}(\bar{\mF}).
	\]
\end{theorem}
For $d=1$ these statements are already known~\cite{MR1102319,MR2908619}. The proof of Theorem~\ref{te:main2} is presented in Section~\ref{sec:prooflocal}.

\subsection*{Notation}

All unspecified limits are taken as $n \to \infty$. We let $\convdis$ denote convergence in distribution.  We call a function $L: \ndR_{>0} \to \ndR_{>0}$ called \emph{slowly varying}, if for all $t>0$
\[
\lim_{x \to \infty} \frac{L(tx)}{L(x)} = 0.
\]
For each $a\in \ndR$ the product of $L(x)x^a$   of a slowly varying function $L$ with a power of~$x$ is called \emph{regularly varying} with index $a$.

Given parameters $0<\alpha \le 2$, $-1 \le \beta \le 1$, $\gamma>0$, and $-\infty < \delta < \infty$ we let $S_\alpha(\gamma, \beta, \delta)$ denote the $\alpha$-stable  distribution with scale parameter  $\gamma$, skewness parameter $\beta$, and location parameter $\delta$, so that the characteristic function of a $S_\alpha(\gamma, \beta, \delta)$-distributed random variable $X_\alpha(\gamma, \beta, \delta)$ is given by
\[
\Exb{e^{- i t X_\alpha(\gamma, \beta, \delta)}} = \begin{cases}
	\exp\left( - \gamma^\alpha |t|^\alpha \left(1 - i \beta  \mathrm{sgn}(t) \tan\left(\frac{\pi \alpha}{2}\right)\right) + i \delta t \right), &\alpha \ne 1 \\
	\exp\left(- \gamma |t| \left( 1 + i \beta \mathrm{sgn}(t) \frac{2}{\pi}  \log( |t|)\right) + i \delta t\right), &\alpha = 1.
\end{cases}
\]
We refer the reader to~\cite{2011arXiv1112.0220J} for a survey on their properties.






\section{Glueing rescaled pointed spaces at their roots}

\label{sec:glue}

In this section we formalize the construction of the limiting space that was informally described in the introduction.

We let $\mathbb{K}$ denote the class of compact pointed metric spaces equipped with Borel probability measures. Here ``pointed'' means that each metric spaces has a distinguished point, called its root point. The class $\mathbb{K}$ is equipped with the pointed Gromov--Hausdorff--Prokhorov distance. A detailed expositions of this concept can be found in \cite{zbMATH06610054}, \cite[Ch. 7]{MR1835418},~\cite[Ch. 27]{zbMATH05306371},~\cite{MR3035742}, \cite[Sec. 6]{MR2571957}, and~\cite{janson2020gromovprohorov}.

Let $(a_i)_{i \ge1}$ denote a sequence of non-negative real numbers and let $(b_i)_{i \ge 1}$ be an element of the set
\[
\Delta = \left\{ (b_i)_{i \ge 1} \in \ndR_{>0}^\infty \,\,\Big\vert\,\, \sum_{i=1}^\infty b_i =1 \right\}.
\]
Let $((X_i, x_i), d_i, \mu_i)_{i \ge 1}$ denote a sequence in $\mathbb{K}$. That is, for each integer $i \ge 1$ we are given a compact metric space $(X_i, d_i)$ equipped with a Borel probability measure $\mu_i$ on $X_i$ and a distinguished point $x_i \in X_i$. 

We can form a measured pointed compact metric space $((X,x), d, \mu)$ by glueing together the rescaled measured spaces $(X_i,  a_i d_i, b_i \mu_i )$ at the points $(x_i)_{i \ge 1}$, and declaring the resulting gluing point $x$ as marked. Let us call this space the \emph{gluing of $((X_i, x_i), d_i, \mu_i)_{i \ge 1}$ rescaled by $(a_i)_{i \ge 1}$ and $(b_i)_{i \ge 1}$}. 
	
In order to make this construction formal, we may assume without loss of generality that the sets $X_i$ are pairwise disjoint and let $X = \bigcup_{i \ge 1} X_i / \sim$ for the smallest equivalence relation $\sim$ satisfying $x_i \sim x_j$ for all $i,j \ge 1$. For each $i \ge 1$ let  $\mathrm{can}_i:  X_i \to X, u \mapsto \overline{u}$ denote the canonical embedding.
For $u \in X_i$ and $v \in X_j$ we set $d(\bar{u},\bar{v}) = a_i d_i(u,v)$ if $i=j$, and $d(\bar{u}, \bar{v}) = a_i d_i(u,x_i) + a_j d_j(x_j, v)$ if $i \ne j$.  We view this space as marked at the point $x = \overline{x}_1$. For each Borel subset $B$ of $X$ we define
$\mu(B) = \sum_{i \ge 1} b_i \mu_i( \mathrm{can}_i^{-1}(B))$. Since $\sum_{i \ge1}b_i = 1$  it follows that $\mu$ is a probability measure. 

\begin{proposition}
	\label{pro:compact}
	If the diameters $\Di(X_i) = \sup_{x,y \in X_i} d_i(x,y)$ of the individual spaces  satisfy 
		$\lim_{i \to \infty} a_i \Di(X_i) = 0$, then the space $X$ is compact.
\end{proposition}
\begin{proof}
Let $C$ be an open cover of $X$. Pick an element $A \in C$ that contains the marked point $x$ of $X$. Then $A$ has an $\epsilon$-neighbourhood of $x$ as a subset for some $\epsilon>0$. For any $i \ge 1$ with $a_i \Di(X_i)< \epsilon$ it follows that $\mathrm{can}_i(X_i) \subset A$. By the assumption $\lim_{i \to \infty} a_i \Di(X_i) = 0$ it follows that there exists $i_0 \ge 1$ such that $\mathrm{can}_i(X_i) \subset A$ for all $i > i_0$. Since $C$ covers the compact subset $\bigcup_{i=1}^{i_0} \mathrm{can}_i(X_i)$, it follows that there exists a finite subset $C' \subset C$ that covers $\bigcup_{i=1}^{i_0} \mathrm{can}_i(X_i)$. Hence $C' \cup \{A\}$ is a finite cover of~$X$.
\end{proof}

Suppose that $\alpha>0$ and $\theta>0$, and let $V_1>V_2> \ldots$ denote the ordered points of $\mathrm{PD}(\alpha, \theta)$. As shown by Pitman and Yor~\cite[Prop. 10, Prop. 14]{MR1434129}, the limit $\lim_{n \to \infty} n^{1/\alpha} V_n$ exists almost surely.  

Given $1 < \lambda \le 2$, suppose that   $(X_i)_{i \ge 1}$ are independent copies of the $\lambda$-stable tree. Then by~\cite[Thm. 1.5]{MR3634265} (see also~\cite{MR3651047})
\[
	\Pr{\He(X_1) > x} \sim c_\beta x^{1 + \lambda/2} \exp(-(\lambda-1)^{1/(\lambda-1)}x^\lambda)
\]
as $x \to \infty$ for some constant $c_\beta >0$.

Hence  by Borel-Cantelli it holds almost surely that  $\Di(X_i) \le 2 \He(X_i) \le  \log (i)$ for almost all $i \ge 1$. In particular, Proposition~\ref{pro:compact} ensures that  for all $s>0$ the random space $\cS(\alpha, \theta, s, L_{\lambda})$ is almost surely compact, with $L_{\lambda}$ referring to the law of the $\lambda$-stable tree.

\section{Dilute Gibbs partitions and the two-parameter Poisson--Dirichlet distribution}

We recall facts on the Gibbs partition model~\cite{MR2245368} and  a connection to the two-parameter Poisson--Dirichlet distribution in a dilute regime~\cite{gibbsmcta,2021arXiv210303751B}, facilitated by calculations of~\cite{MR2597584}.

\subsection{The Gibbs partition model}

Let $S$ be a given finite non-empty set. For our purposes, we define a \emph{partition} of $S$ as a multi-set $P$ of subsets of $S$ that are pairwise disjoint and whose union equals $S$.  We explicitly allow one or several of the members of $S$ to be empty, which is why the partition $P$ is a multi-set and not a regular set. The elements of $P$ are its \emph{components}, and $k \ge 1$ is hence its number of components. We may form the (infinite) collection $\mathrm{Part}(S)$ of all partitions of $S$.

Let $\bm{v} = (v_i)_{i \ge 1}$ and $\bm{w} = (w_i)_{i \ge 0}$ be two sequences of non-negative real numbers such that $w_0 < 1$. This allows us to assign a \emph{weight}
\[
u(P) =  |P|!v_{|P|} \prod_{Q \in P} |Q|!w_{|Q|} 
\]
to each partition $P \in \mathrm{Part}(S)$. Let $n \ge 1$ denote the number of elements of $S$. We define the  \emph{partition function}   by
\[
u_n =  \frac{1}{n!}\sum_{P \in \mathrm{Part}(S)} u(P).
\]
Since $w_0<1$ we have $u_n \in [0, \infty[$. Note also that $u_n$ only depends on the cardinality $n$ of the set $S$.
It is easy to see that the \emph{generating series} $V(x) = \sum_{i \ge 1} v_i x^i$, $W(x) = \sum_{i \ge 0} w_i x^i$, and $U(x) = \sum_{i \ge 0} u_i x^i$ satisfy
\[
U(x) = V(W(x))
\]
In the literature on analytic combinatorics, this relation is called a \emph{composition schema}. We let $\rho_u, \rho_v, \rho_w \in [0, \infty]$ denote the radii of convergence of the series $U(x)$, $V(x)$, and $W(x)$. 

For any integer $n \ge 1$ with $u_n>0$ the  \emph{Gibbs partition} model associated to the weight sequences $\bm{v}$ and $\bm{w}$ describes a random element $P_n$ of   $\mathrm{Part}([n])$ for $[n] := \{1, \ldots, n\}$, with 
\[
\Pr{P_n = P} = \frac{u(P)}{n!u_n}
\]
for each partition $P \in \mathrm{Part}([n])$. We let $N_n$ denote the number of components of $P_n$, and $K_{(1)} \ge K_{(2)} \ge \ldots$ their sizes in nonincreasing order.

\subsection{The dilute regime}

The work~\cite{2021arXiv210303751B} studies Gibbs partitions using analytic methods, with a focus on a regime where the number of components scales at the order $n^\alpha$ for a parameter $0< \alpha <1$. An approach to their study using probabilistic methods is given  in~\cite{gibbsmcta}. The following statement corresponds to~\cite[Thm. 3.13]{gibbsmcta}. Compare with~\cite[Thm. 4.1]{2021arXiv210303751B} that uses slightly different assumptions.

\begin{proposition}
	\label{pro:dilute1}
	Suppose that $\rho_v = W(\rho_w)$. Furthermore, suppose that
	\[
	v_n = L_v(n) n^{-\beta -1}\rho_v^{-n} \qquad \text{and} \qquad w_n \sim c_w n^{- \alpha -1} \rho_w^{-n}
	\]	
	for a slowly varying function~$L_v$, a constant $c_w>0$, and exponents $0<\alpha, \beta <1$. Let $f$ denote the density function of the stable distribution $S_\alpha(\gamma, 1,0)$ for
	\[
	\gamma= \left(\frac{c_w}{W(\rho_w)\alpha} \Gamma(1- \alpha) \cos \frac{\pi \alpha}{2} \right)^{1/\alpha}.
	\]
	Then
	\begin{align}
		\label{eq:diluteclt}
		\frac{N_n}{n^{\alpha}} \convdis Z
	\end{align}
	for a random variable $Z>0$ with density function
	\[
	\tilde{f}(x) = \frac{1}{\alpha \Ex{(X_{\alpha}(\gamma, 1,0))^{\alpha\beta}}} \frac{f\left(\frac{1}{x^{1/\alpha}}\right)}{x^{1 + \beta +1/\alpha}}.
	\]
	For any constant $\delta>0$ we have a local limit theorem
	\begin{align}
		\label{eq:dilutellt}
		\lim_{n \to \infty} \sup_{\ell \ge \delta n^{\alpha}} \left|n^{\alpha} \Pr{ N_n = \ell} - \tilde{f}(\ell / n^{\alpha})  \right| =0.
	\end{align}
\end{proposition}
An important ingredient in the proof are results by Doney~\cite[Thm. 3]{MR1440141} and Bloznelius~\cite[Thm. 1, (iv) and Eq. (133)]{MR4038060} which ensure
\begin{align}
	\label{eq:unasymptotics}
	u_n  \sim n^{-1-\alpha \beta } L_v(n^\alpha) \alpha \Ex{(X_{\alpha}(\gamma, 1,0))^{\alpha \beta}} \rho_w^{-n}.
\end{align}
By~\cite[Thm. 5.1, Ex. 5.5]{2011arXiv1112.0220J} the factor involving the moments of an $\alpha$-stable random variable in Equation~\eqref{eq:unasymptotics}, which also appears in the expression for the density $\tilde{f}(x)$, may be evaluated to
\begin{align*}
\alpha \Ex{(X_{\alpha}(\gamma, 1,0))^{\alpha \beta}}  = 	\frac{\Gamma(1 - \beta) \alpha^{1 - \beta} \left( \frac{c_w}{W(\rho_w)} \Gamma(1-\alpha) \right)^\beta}{\Gamma(1- \alpha\beta)}
\end{align*}
As explained in~\cite[Sec. 3.4]{gibbsmcta}, in the setting of Proposition~\ref{pro:dilute1} the density of $N_n$ admits the expression
\begin{align}
	\label{eq:extref}
	\Pr{N_n = \ell} = \frac{v_\ell \rho_v^\ell}{ u_n \rho_w^n} \Pr{S_\ell = n},
\end{align}
with $S_\ell$ denoting the sum of $\ell$ independent copies of a random variable with probability generating series $W(\rho_w z) / W(\rho_w)$. The result~\cite[Lem. 4 (i), Eq. (3.14)]{MR1440141} states precise expressions and bounds for $\Pr{S_\ell = n}$ when $\ell$ grows faster than $n^\alpha$, which imply:

\begin{proposition}
	\label{prop:tight445}
	Under the same assumptions of Proposition~\ref{pro:dilute1}:
	\begin{enumerate}
		\item For each $A>0$ there exist $B>0$ such that 
			\[
				\Pr{S_\ell = n} \le \exp(-B \ell)
			\]
			uniformly for $\ell \ge A n$.
		\item There exist $A, B, C, D>0$ such that
		\[
			\Pr{S_\ell = n} \le  C \sqrt{\ell^{\frac{1}{1-\alpha}} n^{-\frac{2 - \alpha}{1-\alpha} }} \exp\left(- D \ell^{\frac{1}{1-\alpha}} n^{- \frac{\alpha}{1- \alpha}}\right)
		\]
		uniformly for $An^{\alpha} \le \ell \le B n$.
	\end{enumerate}
\end{proposition}
In fact,~\cite[Lem. 4]{MR1440141} makes  more precise statements, but we won't require the  expressions and notations for the involved constants.

\paragraph*{Extremal component sizes}
Suppose that the assumptions of Proposition~\ref{pro:dilute1} hold, and define the following point process on $]0,1]$
\[
\Upsilon_n = \sum_{\substack{1 \le i \le N_n \\ K(i) > 0}} \delta_{K(i) / n},
\]
with $\delta$ referring to the Dirac measure. It was shown in~\cite[Cor. 3.15]{gibbsmcta} that 
\begin{align}
	\label{eq:upsilon}
\Upsilon_n \convdis \Upsilon_{\alpha, \beta} 
\end{align}
as $n \to \infty$, for a point process $\Upsilon_{\alpha, \beta}$ on $]0,1]$ that only depends on $\alpha$ and $\beta$. At the time of writing the paper~\cite{gibbsmcta}, the author  did not recognize the distribution of $\Upsilon_{\alpha, \beta}$, and merely observed basic properties. For example, that almost surely $\Upsilon_{\alpha, \beta}$ has infinitely many points that sum up to~$1$, and that the  distribution function of the $k$th largest point may be determined~\cite[Prop. 3.16]{gibbsmcta}. In the proof of~\cite[Cor. 3.15]{gibbsmcta}, it was also observed that for each  integer $m \ge 1$ the $m$th correlation function of $\Upsilon_{\alpha, \beta}$ is given by
\begin{align*}
	   \Exb{ \left(\mathrm{Poi}\left( \frac{c_w}{W(\rho_w)}Z\right)\right)_m }  \one_{\substack{{x_1 + \ldots + x_m \le 1} \\ {x_1, \ldots, x_m \ge 0}}} \frac{(1-x_1 - \ldots - x_m)^{\alpha(m-\beta)-1}}{x_1^{\alpha+1}\cdots x_m^{\alpha+1}  },
\end{align*}
with the first factor denoting the $m$th factorial moment of the Poisson random variable $\mathrm{Poi}\left( \frac{c_w}{W(\rho_w)}Z\right)$ with random parameter $\frac{c_w}{W(\rho_w)}Z$, so that 
\[
\Pr{\mathrm{Poi}(\upsilon Z) = k} = \Exb{ \frac{1}{k!} \left(\frac{c_w}{W(\rho_w)} Z\right)^k \exp\left(-\frac{c_w}{W(\rho_w)} Z\right) }, \qquad k \ge 0.
\]
The distribution of $\frac{c_w}{W(\rho_w)}Z$ in fact only depends on $\alpha$ and $\beta$, so that
\begin{align*}
	\Exb{ \left(\mathrm{Poi}\left( \frac{c_w}{W(\rho_w)}Z\right)\right)_m } 
	= \left( \frac{ \alpha } {\Gamma(1- \alpha)} \right)^m  \frac{\Gamma(1 - \alpha \beta) \Gamma(m + 1 - \beta)}{\Gamma(1 + \alpha(m-\beta)) \Gamma(1 - \beta)}.
\end{align*}
The author is grateful to  Markus Heydenreich for pointing out a connection to the two-parameter Poisson--Dirichlet process $\mathrm{PD}(\alpha, \theta)$ introduced by Pitman and Yor~\cite{MR1434129}: Handa~\cite[Thm. 2.1]{MR2597584} determined the correlation functions of $\mathrm{PD}(\alpha, \theta)$. For $\theta= -\alpha \beta$, they agree with those of $\Upsilon_{\alpha, \beta}$. Hence, using the method of moments, we may apply~\cite[Thm. 3.1]{MR818219} to deduce
\begin{align}
	\label{eq:pd}
\Upsilon_{\alpha, \beta} \eqdist \mathrm{PD}(\alpha, -\alpha \beta).
\end{align}

We are going to use the convergence~\eqref{eq:upsilon} of the discrete model to verify a result about concatenations of this process. Compare also with a result on concatenations~\cite[Prop. 22]{MR1434129} for a different range of parameters.

\begin{lemma}
	\label{le:combine}
	Let $0< \alpha_1, \alpha_2, \alpha_3 < 1$. Let $X_1 > X_2 > \ldots>0$ denote the ranked points of $\Upsilon_{\alpha_1 \alpha_2,  \alpha_3}$ and for each integer $i \ge 1$ let $Y_{i,1} > Y_{i,2} > \ldots>0$ denote an independent copy of the ranked points of $\Upsilon_{\alpha_1, \alpha_2}$. Then $(X_i Y_{i,j})_{i, j \ge 1}$ is distributed like the points of $\Upsilon_{\alpha_1, \alpha_2 \alpha_3}$.
\end{lemma}
\begin{proof}
	Consider power law probability weights $r_{i,n} \sim n^{-\alpha_i -1}$ for $i=1,2,3$ as $n \to \infty$ with probability generating functions $R_i(z) := \sum_{n \ge 0} r_{i,n} z^n$. Our proof strategy is to determine the limit of the extremal $R_3$-component sizes in the composition scheme $R_3(R_2(R_1(z)))$ in two ways that are equivalent by the associative law.
	
	If we set $V(z) = R_3(R_2(z))$ and $W(z) =R_1(z)$, then the $n$-th coefficient satisfies by~\cite[Thm. 1, (iv) and Eq. (133)]{MR4038060} that
	\begin{align*}
		[z^n]V(z) \sim n^{-1-\alpha_2 \alpha_3}  \alpha_2 \Exb{\left(X_{\alpha_2}\left(\left(\frac{\Gamma(1- \alpha_2)}{\alpha_2}  \cos \frac{\pi \alpha_2}{2} \right)^{1/\alpha_2}, 1,0\right)\right)^{\alpha_2 \alpha_3}}.
	\end{align*}
Thus, the assumptions for applying Equation~\eqref{eq:upsilon} are satisfied, yielding that the $R_1$-component sizes in the Gibbs partition associated to the composition scheme $V(W(z))$ rescaled by $n^{-1}$ admit the point process $\Upsilon_{\alpha_3, \alpha_1 \alpha_2}$ as limit.

If instead we set $V(z) = R_3(z)$ and $W(z) = R_2(R_1(z))$, then the coefficients of $W(z)$ are probability weights of an asymptotic power-law with exponent $-\alpha_1\alpha_2-1$, and we obtain that the $R_2(R_1)$-component sizes in the  Gibbs partition associated to the composition scheme $V(W(z))$ admit after rescaling by $n^{-1}$ the point process $\Upsilon_{\alpha_1 \alpha_2, \alpha_3}$ as limit. Each of these large $R_2(R_1)$-components splits into $R_1$-components according to the Gibbs partition model associated to the composition scheme $R_2(R_1(z))$. Hence, by applying~\eqref{eq:upsilon} for a second time, each of the $R_2(R_1)$-components admits a copy of $\Upsilon_{\alpha_1, \alpha_2}$  as limit after rescaling by its size. Thus the $R_1$-component sizes obtained in this way admit after rescaling by $n^{-1}$ the joint distributional limit $(X_i Y_{i,j})_{i, j \ge 1}$.

By the associative law, the two limiting distributions must agree. Thus, the family $(X_i Y_{i,j})_{i, j \ge 1}$ is distributed like the points of $\Upsilon_{\alpha_1, \alpha_2 \alpha_3}$.
\end{proof}

\section{Scaling limits of Kemp's $d$-dimensional binary trees}

\label{sec:mainproof}

We start by proving the tail-bounds for the height.

\begin{proof}[Proof of Theorem~\ref{te:bound}]
	The proof is by induction on $d$. For $d=1$, such bounds have already been established in~\cite{MR3077536} in the context of branching processes. See also~\cite{addarioberry2022random} for a generalization to bounds requiring no assumptions on the offspring distribution. For $d \ge 2$ we let $1 \le N_n \le n/2$ denote the number of vertices of the first level of $\mF_{d,n}$, and $K_{(1)} \ge \ldots \ge K_{(N_n)} \ge 1$ the  number of vertices in the trees attached to the first level. By induction hypothesis, there exist constants $C, c>0$ independent of $n$ and $x$ such that for $1 \le i \le d-1$
	\[
	\Pr{\He(\mF_{i,n}) \ge x} \le C \exp(-c x^2/n).
	\]
	If $\He(\mF_{d,n}) \ge x$, then the first-level has height at least $x/2$ or one of the trees attached to it has height at least $x/2$. Let us denote these two events by $\cE_1$ and $\cE_2$. Conditioned by its size, the first level is distributed like a uniform binary tree. Hence, summing over the number of vertices in the first level, we obtain using the induction hypothesis
	\begin{align*}
		\Pr{\cE_1} &= \sum_{\ell=1}^{\lfloor n/2 \rfloor} \Pr{N_n = \ell} \Pr{\He(\mF_{1,\ell}) > x/2} \\
		&\le \sum_{\ell=1}^{\lfloor n/2 \rfloor} \Pr{N_n = \ell} C \exp(-c x^2/4\ell) \\
		&\le C \exp(-c x^2 / 4n).
	\end{align*}
	Likewise, summing over the number of vertices in the second level, we obtain by the induction hypothesis
	\begin{multline}
		\label{eq:klemptner}
		\Pr{\cE_{2}}  =  \sum_{\ell=1}^{\lfloor n/2 \rfloor} \Pr{N_n = \ell} \sum_{ \substack{ k_1 \ge \ldots \ge k_\ell \ge 1 \\ k_1 + \ldots + k_\ell = n - \ell}} \Pr{K_{(1)} = k_1, \ldots, K_{(\ell)} = k_\ell } p_{k_1, \ldots, k_\ell},
	\end{multline}
	with $p_{k_1, \ldots, k_\ell}$ denoting the probability that in a forest of $\ell$ random $(d-1)$-dimensional supertrees with number of vertices $\bm{k} = (k_1, \ldots, k_\ell)$ at least one of them has height at least $x/2$. We set $\|\bm{k}\|_1 = \sum_{i=1}^\ell k_i$ and of course in the sum above we only consider indices for which $\|\bm{k}\|_1 = n-\ell$. Using the union bound and the induction hypothesis we obtain 
	\begin{align*}
		p_{k_1, \ldots, k_\ell} &\le  \sum_{i=1}^\ell \Pr{\He(\mF_{d-1,k_i}) > x/2} \\
		&\le C \sum_{i=1}^\ell \exp(-c x^2/4k_i) \\
		&= C \sum_{i=1}^\ell \left( \exp(-c x^2/4\|\bm{k}\|_1) \right)^{1/q_i},
	\end{align*}
	with $q_i = k_i / \|\bm{k}\|_1$ so that $q_1 + \ldots + q_\ell= 1$. We hence seek to maximize a function \[
	f_\ell(x_1, \ldots, x_\ell) = \sum_{i=1}^\ell A^{1/x_i}
	\]
	for $A =  \exp(-c x^2/(4\|\bm{k}\|_1))$ on the set 
	\[
	M_\ell = \{ (x_i)_{1 \le i \le \ell} \in \ndR_{>0}^\ell  \mid \sum_{i=1}^\ell x_i = 1\}.
	\]
	
		Let us, for now, consider the case where $x$ is large enough so that
	\begin{align}
		\label{eq:bandaid}
		c x^2/(4\|\bm{k}\|_1) > 2.
	\end{align}
	The function $f_\ell$ is continuous on $\ndR_{\ge 0}^\ell$  and hence assumes a maximum on the closed hull $\overline{M}_\ell$. Let $\bm{y} \in \overline{M}_\ell$ denote a point where this maximum is assumed. If $\bm{y} \in M_\ell$, then by Lagrange multipliers it follows that $\bm{y} = (1/\ell, \ldots, 1 / \ell)$. This is because 
	\[
		\frac{\mathrm{d}^2}{\mathrm{d} q^2}A^{1/q} = \frac{A^{1/q}}{q^4} \log(A) (2q + \log(A)) > 0
	\]
	for $0<q<1$ due to Inequality~\eqref{eq:bandaid}, which ensures $\log(A) < -2$. If $\bm{y} \notin M_\ell$, then the vector $\tilde{\bm{y}}$, obtained from $\bm{y}$ by deleting all zero coordinates, lies in $M_{\tilde{\ell}}$ for some $1 \le \tilde{\ell} < \ell$. Since $\sup_{ \tilde{\bm{x}} \in \overline{M}_{\tilde{\ell}}} f_{\tilde{\ell}}(\tilde{\bm{x}}) \le \sup_{\bm{x} \in \overline{M}_\ell } f_{\ell}(\bm{x})$, it follows that $f_{\tilde{\ell}}$ assumes its maximum on $\overline{M}_{\tilde{\ell}}$ at the point $\tilde{\bm{y}} \in M_{\tilde{\ell}}$, and clearly $f_{\tilde{\ell}}(\tilde{\bm{y}}) = f_{\ell}(\bm{y})$. By Lagrange multipliers it follows that $\tilde{\bm{y}} = (1/\tilde{\ell}, \ldots, 1/\tilde{\ell})$. 
	
	Hence it follows that 
	\begin{align}
		\label{eq:tight}
		p_{k_1, \ldots, k_\ell} &\le C  \sup_{1 \le i \le \ell} i  \exp(-c i x^2/\|\bm{k}\|_1).
	\end{align}
	Using  Inequality~\eqref{eq:bandaid}, it follows that there exists a constant $c_1>0$ that does not depend on $\ell$ or $k_1, \ldots, k_\ell$ such that 
	\begin{align*}
		p_{k_1, \ldots, k_\ell} &\le C   \exp(-c_1 x^2/\|\bm{k}\|_1)
	\end{align*}
	uniformly for all $x>0$ satisfying Inequality~\eqref{eq:bandaid}. For all $x<0$ that do not satisfy  Inequality~\eqref{eq:bandaid} we have $\exp(-c_1 x^2/\|\bm{k}\|_1) \ge \exp(- (c_1/c) 2)$. Hence, setting $C_1 = \max(C,  \exp( (c_1/c) 2))$ it follows using $	p_{k_1, \ldots, k_\ell} \le 1$ that
	\begin{align}
		\label{eq:realtight}
	p_{k_1, \ldots, k_\ell} &\le C_1   \exp(-c_1 x^2/\|\bm{k}\|_1)
\end{align}
	uniformly for all $x>0$.
	Hence, by Equation~\eqref{eq:klemptner}
	\[
	\Pr{\cE_{2}}  \le C_1   \exp(-c_1 x^2/n).
	\]
	Consequently, there exist constants $C_2, c_2>0$ that do not depend on $n$ or $x$ with
	\[
	\Pr{ \He(\mF_{d,n}) \ge x } \le \Pr{\cE_1} + \Pr{\cE_2} \le C_2   \exp(-c_2 x^2/n).
	\]
	This completes the proof.
\end{proof}

We are ready to prove our main result. As a byproduct, we recover Kemp's asymptotic formula~\eqref{eq:fdasymp} via probabilistic methods.

\begin{proof}[Proof of Theorem~\ref{te:main}]
	Let us recall some facts stated in the introduction. The equation
	\[
		F_1(x) = x(1 + 2F_1(x) + F_1(x)^2)
	\]
	may be solved for $F_1(x)$, yielding the Catalan numbers
	\[
		f_{1,n} = \frac{1}{n+1} \binom{2n}{n} \sim  \frac{1}{\sqrt{\pi}} 4^n n^{-3/2}
	\]
	and $F_1(1/4) = 1$.
	By the pioneering  invariance principle for large simply generated trees~\cite{MR1207226}, the Brownian tree is the scaling limit of $\mF_{1,n}$ after rescaling distances by $2^{-3/2} n^{-1/2}$. 
	
	Let us now proceed by induction on $d \ge 2$.  The generating series \[
	F_d(x) = F_1(xF_{d-1}(x))
	\] corresponds to the Gibbs partition model with $V(x) = F_1(x)$ and $W(x) = x F_{d-1}(x)$, so that $v_n = f_{1,n}$ and $w_n = f_{d-1,n-1}$. Using the induction hypotheses 
	$
		f_{d-1,n} \sim \frac{4^{1- 2^{-(d-1)}}}{|\Gamma(-2^{-(d-1)})|} 4^n n^{-1 -2^{-(d-1)}}
	$
	and $F_{d-1}(1/4) = 1$
	and the first order asymptotics for $f_{1,n}$ 
	it follows that  the assumptions of Proposition~\ref{pro:dilute1} are satisfied with $\alpha=2^{-(d-1)}$,  $\beta=1/2$, $L_v = \frac{1}{\sqrt{\pi}}$, $c_w = \frac{4^{ (- 2^{-(d-1)})}}{|\Gamma(-2^{-(d-1)})|}$, $\rho_w = 1/4$, and $W(\rho_w) = 1/4$. Equation~\eqref{eq:unasymptotics} hence yields
	\begin{align*}
		f_{d,n}   &\sim n^{-1-\alpha \beta } L_v(n^\alpha) \frac{\Gamma(1 - \beta) \alpha^{1 - \beta} \left( \frac{c_w}{W(\rho_w)} \Gamma(1-\alpha) \right)^\beta}{\Gamma(1- \alpha\beta)} \rho_w^{-n} \\
		&= n^{-1-2^{-d}} \frac{2^{2(1- 2^{-d})}}{|\Gamma(-2^{-d})|}  4^{-n}.
	\end{align*}
	Thus, we recover Equation~\eqref{eq:fdasymp}.
	
	Furthermore, by our induction hypothesis the random tree $\mF_{d-1,n}$ admits a scaling limit after rescaling distances by $2^{-3/2} n^{-1/2}$. Let $L$ denote the law of that scaling limit.  Since Proposition~\ref{pro:dilute1} applies to $V(W(x))$ it follows that the number $N_n$ of vertices in the first level of $\mF_{d, n}$  satisfies $N_n / \sqrt{n} \convdis Z>0$. Moreover, the first level of $\mF_{d,n}$ is distributed like a random tree with size $N_n$, whose diameter $D_{N_n}$ has the property that $ D_{N_n} / (2^{3/2} N_n^{1/2})$ converges in distribution towards the diameter of the Brownian tree. Thus, the diameter of the first level is $O_p(n^{1/4})$ and its number of vertices is $O_p(n^{1/2})$. Hence, if we rescale $\mF_{d,n}$ by $2^{-3/2} n^{-1/2}$ then the first level contracts to a single point with no mass.
	
	Let $X_1 > X_2 > \ldots>0$ denote the ordered points of $\Upsilon_{\alpha, \beta} \eqdist \mathrm{PD}(\alpha, -\alpha \beta)$, with $\alpha = 1/2^{d-1}$ and $\beta= 1/2$. Let $K_{(1)} \ge K_{(2)} \ge \dots$ denote the numbers of vertices in the largest tree-components $T_{(1)}, T_{(2)}, \ldots$ attached to the first level of $\mF_{d,n}$.  By Equation~\eqref{eq:upsilon} we know that 
	\begin{align}
		\label{eq:kin}
		(K_{(i)}/n)_{i \ge 1} \convdis (X_i)_{i \ge 1}.
	\end{align}
Jointly for all $i \ge 1$ we know by induction hypothesis that $T_{(i)}$ equipped with the uniform measure on its vertex set is distributed like $F_{d-1, K_{(i)}}$ and hence admits an independent sample $\tau_i$ of the law $L$ as  Gromov--Hausdorff--Prokhorov scaling limit $\tau_i$ after rescaling distances by $2^{-3/2} K_{(i)}^{-1/2}$. Thus, the limit after rescaling by $2^{-3/2} n^{-1/2}$ instead is distributed like $\tau$ with distances multiplied by the independent random variable $\sqrt{X_i}$. This convergence holds jointly for all $i \ge 1$ and  the first level  of $\mF_{d,n}$ contracts to a single point with no mass. Thus, in order to deduce that
\begin{align}
	\label{eq:step2}
	2^{-3/2} n^{-1/2}\mF_{n,d} \convdis \cS(1/2^{d-1}, -1/2^d, 1/2, L)
\end{align}
in the Gromov--Hausdorff--Prokhorov sense all that remains is to verify  tightness, which boils down to showing  for each $\epsilon>0$
\begin{align}
	\label{eq:step1}
	\lim_{k \to \infty} \limsup_{n \to \infty} \Prb{\sup_{j \ge k} \He\left(T_{(j)}\right) > \epsilon \sqrt{n}} = 0.
\end{align}
To this end, note that Inequality~\eqref{eq:realtight} implies that conditionally on  $(K_{(i)})_{i \ge 1}$ the probability for the event $\sup_{j \ge k} \He\left(T_{(j)}\right) > \epsilon \sqrt{n}$ is bounded by \[
C_1 \exp\left(-c_1  \epsilon^2 n / \sum_{i \ge k} K_{(i)} \right).
\] By Equation~\eqref{eq:kin} and $\sum_{i \ge k} K_{(i)} = n - N_n - \sum_{i=1}^{k-1} K_{(i)}$ with $N_n = o_p(n)$  we know that for each fixed $k$
\[
	n^{-1} \sum_{i \ge k} K_{(i)} \convdis \sum_{i \ge k} X_i.
\]
Since $\sum_{i \ge k} X_i$ tends in probability to zero as $k \to \infty$ this verifies Equation~\eqref{eq:step1} and hence also Equation~\eqref{eq:step2}.

It remains to explicitly describe the limit in~\eqref{eq:step2}. For $d=2$ we have that $L$ is the law $L_{\mathrm{Brownian}}$ of the Brownian tree, so the limit of $2^{-3/2} n^{-1/2}\mF_{n,2}$ is $\cS(1/2, -1/4, 1/2, L_{\mathrm{Brownian}})$. By induction hypothesis, we may hence assume that for $d>3$ we have that $L$ is the law of $\cS(1/2, -1/2^{d-1}, 1/2, L_{\mathrm{Brownian}})$. By Lemma~\ref{le:combine} and Equation~\eqref{eq:pd}  it follows that 
\[
\cS(1/2^{d-1}, -1/2^d, 1/2, L) \eqdist \cS(1/2, -1/2^d, 1/2, L_{\mathrm{Brownian}}).
\]
This completes the proof.
\end{proof}

\section{Sketching a  phase diagram}
\label{sec:phase}

The diversity of tree structures and models in probabilistic combinatorics is so vast that a complete phase diagram for supertrees is beyond the scope of this work. Instead, we describe some phases with fundamentally different behaviour and show that each phase has some level of universality with natural archetypes.

\label{sec:phases}

\subsection{The dilute regime}

 The method in the proof of Theorem~\ref{te:main} is not confined to the model under consideration. It is easily transferable to settings of composition schemes $V(W(x))$ of weighted classes   $V$ and $W$ of trees where the assumptions of Proposition~\ref{pro:dilute1} are met for some $0<\alpha, \beta < 1$, such that random $n$-sized trees from the class $W$ admit a Gromov--Hausdorff--Prokhorov limit after multiplying distances by $n^{-\delta}$ for some $0<\delta<1$, and random trees from the class $V$ with size roughly $n^\alpha$ have diameter $o_p(n^{\delta})$. This way, as in the proof of Theorem~\ref{te:main}, the first level in a large random supertree contracts to a point with no mass, and the global shape is determined by the largest tree components in the second level which scale at the order $n$ with joint fluctuations determined by an $(\alpha, - \alpha \beta)$-Poisson--Dirichlet process. Furthermore, analogously as in the proof of Theorem~\ref{te:main}, higher dimensional supertrees may be treated by induction on their dimension. Moreover, at no point we actually use that $V$ and $W$ are classes of trees. The same construction and methods work for superstructures when $V$ and $W$ are weighted classes of graphs or maps on a surface.

\subsubsection{Iterated unordered labelled trees}
Consider the class $T_1$ of labelled rooted unordered trees, so that its exponential generating series $T_1(z) = \sum_{n \ge 1} \frac{n^{n-1}}{n!}z^n$ satisfies
\[
	T_1(z) = z \exp(T_1(z)).
\]
By Aldous scaling limit~\cite{MR1085326,MR1207226} the uniform random $n$-vertex tree from the class $T_1$ admits the Brownian tree as scaling limit after multiplying distances by $1 / (2 \sqrt{n})$. For each $d \ge 2$ let $T_d$ denote the class of $d$-dimensional supertrees obtained by iterating the class $T_1$, so that
\[
	T_d(z) = T(zT_{d-1}(z)).
\]
Since $T_1(1/e) = 1$, it follows that $T_d(1/e) = 1$ for all $d \ge 1$.  Using Equation~\eqref{eq:unasymptotics} we obtain inductively that the composition scheme $T_1(zT_{d-1}(z))$ with $V(z) = T_1(z)$ and $W(z) = z T_{d-1}(z)$ satisfies the assumptions of Proposition~\ref{pro:dilute1} with $\alpha=1/2^{d-1}$ and $\beta=1/2$. Hence,
\[
	[z^n] T_d(z) \sim c_d n^{-1-1/2^d} e^n
\]
with, for $c_1 = 1 / \sqrt{2\pi}$ and
\[
	c_d = \frac{\sqrt{2 c_{d-1} |\Gamma(-1/2^{d-1})| } }{|\Gamma(-1/2^d)|}
\]
for $d \ge 2$. Solving this recursion and arguing analogously as in the proof of Theorem~\ref{te:main} and Theorem~\ref{te:bound} we obtain:

\begin{corollary}
	We have
	\[
		[z^n] T_d(z) \sim \frac{2^{1- 1/2^d}}{|\Gamma(-1/2^d)|} n^{-1-2^{-d}} e^n
	\]
	as $n \to \infty$.
	The uniform $n$-vertex supertree $\mT_{d,n}$ from the class $T_d$ equipped with the uniform measure $\mu_{\mT_{d,n}}$ on its vertex set satisfies
	\[
	(\mT_{d,n}, \frac{1}{2 \sqrt{n}} d_{\mT_{d,n}}, \mu_{\mT_{d,n}})  \convdis  \cS(1/2, -1/2^d, 1/2, L_{\mathrm{Brownian}})
	\]
	in the Gromov--Hausdorff--Prokhorov sense as $n \to \infty$. Furthermore, there are constants $C,c>0$ such that
	\[
		\Pr{\He(\mT_{d,n}) \ge x} \le C \exp(-c x^2 /n)
	\]
	uniformly for all $n$.
\end{corollary}

\subsubsection{Kemp's multidimensional simply generated trees}

The following result treats   Kemp's multidimensional simply generated trees~\cite{MR1370950} for specific weights. It serves as an example for supertrees whose scaling limit involves rescaled copies of some stable tree that isn't necessarily the Brownian tree.

Let $d \ge 2$ and for each $i\in \{1,2, \ldots, d\}$  let $\xi_i$ denote a random non-negative integer satisfying $\Pr{\xi_i= 0}>0$ and  $\Pr{\xi_i \ge 2} > 0$. Let $\mT$ denote the random $d$-dimensional  supertree where the first level is given by a $\xi_1$-Bienaym\'e--Galton--Watson tree, and for each $2 \le i \le d$ the $i$th level is obtained by adding for each vertex $v$ of the $(i-1)$th level  an edge between $v$ and an independent copy of a $\xi_i$-Bienaym\'e--Galton--Watson tree. Let $\mT_n$ denote the result of conditioning $\mT$ on having $n$ vertices. Let $\mu_{\mT_n}$ denote the uniform measure on the vertex set of~$\mT_n$.

\begin{theorem}
	\label{te:mulscal}
		 Suppose that for each $1 \le i \le d$ we have $\Ex{\xi_i} = 1$. Let $1<a_i\le 2$ be given. If $a_i\in ]1,2[$ suppose that $\Pr{\xi_i = n} \sim c_i n^{-a_i-1}$ for some constant $c_i>0$. If $a_i=2$ suppose instead that  $\Va{\xi_i}< \infty$. Furthermore, we assume that for all $0 \le i \le d-2$ we have \begin{align}
		 	\label{eq:ineqas}
		 1 - \frac{1}{a_d} > \frac{1}{a_d \cdots a_{d-i}} \left(1 - \frac{1}{a_{d-i-1}}\right).
		 \end{align} Let $L_{a_d}$ denote the law of the $a_d$-stable tree. Then there exists a constant $\kappa>0$ with
		\begin{align}
			\label{eq:mulscal}
	\left(\mT_{n}, \frac{\kappa}{ n^{1 - 1/a_d}} d_{\mT_{n}}, \mu_{\mT_{n}}\right) \convdis  \cS(1/a_d,-1/(a_1 \ldots a_d),1 - 1/a_d,L_{a_d})
	\end{align}
	as $n \to \infty$ with respect to the rooted Gromov--Hausdorff--Prokhorov metric. Furthermore, for each $0<\delta<a_d$  there exist constants $C,c>0$ such that for all $n$ and all $x>0$
	\begin{align}
			\label{eq:mulbound}
	\Pr{\He(\mT_{n}) \ge x} \le C \exp(-c (x/n^{1-1/a_d})^\delta ).
	\end{align}
	If $a_d=2$ then we may even set $\delta=2$.
\end{theorem}
The overall strategy is the same as for Theorems~\ref{te:main} and~\ref{te:bound}. We build upon stable tree scaling limits of large Bienaym\'e--Galton--Watson tree~\cite{MR1964956,MR1954248} and tail-bounds for their height~\cite{MR3077536,MR3651047, addarioberry2022random}.

\begin{proof}[Proof of Theorem~\ref{te:mulscal}]
	The assumptions entail that the probability for a $\xi_i$--Bienaym\'e--Galton--Watson tree to have $n$ vertices grows like a constant times $n^{-1-1/a_i}$, see for example~\cite[Thm. 18.14]{MR2908619}. Furthermore, there exists a constant $b_i>0$ such that the conditioned tree approaches the $a_i$-stable tree after rescaling distances by $\frac{b_i}{ n^{1 - 1/a_i}}$, see~\cite{MR1964956,MR1954248}.
	
	For $d=2$, arguing analogously as in the proof of  Theorem~\ref{te:main} it follows from Proposition~\ref{pro:dilute1} that the largest components in the second level have diameters that scale at the order $n^{1-1/a_2}$, whereas the first level has about $n^{1/a_2}$ vertices and hence its diameter scales at the order $n^{(1/a_2)(1-1/a_1)}$. Our assumption~\eqref{eq:ineqas} ensures that \[
	(1/a_2)(1-1/a_1)<1-1/a_2,
	\]
	so that the diameter of the first level scales at a smaller order than the diameter of the largest components of the second level. 
	
	Similarly, for $d > 2$ the first level has about $n^{1/(a_2 a_3 \cdots a_d)}$ vertices and hence its diameter scales at the order $n^{ (1/(a_2 a_3 \cdots a_d))(1-1/a_1)}$. Thus, if by induction the largest component attached to the first level has a diameter that scales at the order $n^{1-1/a_d}$, then Assumption~\eqref{eq:ineqas} ensures that the diameter of the first level scales at a smaller order. 
	
	Thus, we may prove the scaling limit~\eqref{eq:mulscal} by arguing completely analogously as in the proof of Theorem~\ref{te:main}. The only part where we have to be careful is the analogon of the tightness condition~\eqref{eq:step1}, with $\sqrt{n}$ replaced by $n^{1-1/a_d}$. For $a_d=2$, all arguments here (and also for Inequality~\eqref{eq:mulbound} due to the universal bounds~\cite{addarioberry2022random}) are identical, hence no adaptions are needed. For $1<a_d<2$, we have to put in extra effort to the get the necessary deviation bounds.  Hence, for the remainder of the proof we assume $1<a_d<2$.

	Furthermore, the case $d=1$ was already treated in~\cite{MR1964956,MR1954248} and~\cite{MR3651047}. Hence we assume that $d \ge 2$.

	
	We first aim to verify the bound~\eqref{eq:mulbound} as the intermediate results in the proof are then used in verifying the missing tightness condition for the scaling limit. Let $0 < \delta < a_d$ be given.
	If the tree $\mT_n$ has height at least $x$ then one of the trees in the first level or one of the supertrees attached to it has height at least $x/2$. We denote the corresponding events by $\cE_1$ and $\cE_2$.
	
	 Note that without loss of generality we may assume that 
	\begin{align}
		\label{eq:sanity}
	n^{1-1/a_d} \le x \le n.
	\end{align}
	Let us start with the event $\cE_1$. As in the proof of Theorem~\ref{te:bound}, we may write
		\begin{align}
			\label{eq:stereo1}
		\Pr{\cE_1} &= \sum_{\ell=1}^{\lfloor n/2 \rfloor} \Pr{N_n = \ell} \Pr{\He(\mF_{\ell}) > x/2} 
	\end{align}
with $\mF_\ell$ denoting a $\xi_1$-Bienaym\'e--Galton--Watson tree conditioned on having $\ell$ vertices, and $N_n$ denoting the number of vertices in the first level of $\mT_n$.

 By~\cite[Thm. 4]{MR3651047}, for each $0<\delta_1<a_1$ there exist constants $C_1, c_1>0$ such that
\begin{align}
				\label{eq:stereo2}
\Pr{\He(\mF_{\ell}) > x/2} \le C_1 \exp( -c_1 (x / \ell^{1- 1 / a_1})^{\delta_1} )
\end{align}
uniformly in $x>0$ and $\ell \ge 1$.

We first consider the case  $a_d \ge a_1$. For the bound~\eqref{eq:mulbound} it always suffices to show it for a $\delta$ that is even closer to $a_d$. Hence we may assume that $a_d > \delta \ge a_1$ and we may choose $\delta_1$ sufficiently close to $a_1$ so that
$
a_1 > \delta_1 > a_1 \frac{\delta}{a_d}.
$
Then, using Inequality~\eqref{eq:sanity}
\[
	\frac{n^{(1 - 1/a_d)\delta}}{\ell^{(1- 1/a_1)\delta_1}} \ge \frac{n^{(1 - 1/a_d)\delta}}{n^{(1- 1/a_1)\delta_1}} = n^{\delta - \delta_1 + \frac{\delta_1}{a_1} - \frac{\delta}{a_d}} \ge n^{\delta - \delta_1} \ge x^{\delta - \delta_1}.
\]
This entails
\[
	\left( \frac{x}{\ell^{1-1/a_1}} \right)^{\delta_1} \ge \left( \frac{x}{\ell^{1-1/a_d}} \right)^{\delta}.
\]
Combining this with Equation~\eqref{eq:stereo1} and Inequality~\eqref{eq:stereo2} we obtain
\[
\Pr{\cE_1} \le C_1 \exp( -c_1 (x / n^{1- 1 / a_d})^{\delta} ).
\]

Now, consider the case $a_d < a_1$. Then we may take $\delta_1$ close enough to $a_1$ so that $\delta_1> \delta$. For ease of notation, we set $\alpha = \frac{1}{a_2 \cdots a_d}$. Next, by Inequality~\eqref{eq:ineqas} we may take $\epsilon>0$ small enough (depending only on $a_1, \ldots, a_d$) so that for all $1 \le \ell \le n^{\alpha + \epsilon}$ we have 
\[
	\frac{n^{(1 - 1/a_d)\delta}}{\ell^{(1- 1/a_1)\delta_1}} \ge \frac{n^{(1 - 1/a_d)\delta}}{n^{\left(\alpha + \epsilon\right)(1- 1/a_1)\delta_1}} \ge n^{(1 - 1/a_d)(\delta - \delta_1)} \ge x^{\delta - \delta_1}.
\]
For the last inequality sign we applied $\delta_1 > \delta$ and Inequality~\eqref{eq:sanity}.  Using Equation~\eqref{eq:stereo1} and Inequality~\eqref{eq:stereo2} it follows that
\begin{align}
	\label{eq:blockfabrik}
\Pr{\cE_1} \le C_1  \exp( -c_1 (x / n^{1- 1 / a_d})^{\delta} ) + C_1 \sum_{ \mathclap{n^{\alpha + \epsilon} \le \ell \le \frac{n}{2} }} \Pr{N_n = \ell} \exp( -c_1 (x / \ell^{1- 1 / a_1})^{\delta_1} ). 
\end{align}
By Proposition~\ref{prop:tight445}, Equation~\eqref{eq:unasymptotics} and Equation~\eqref{eq:extref} there exist constants $A,B,C>0$ such that
\begin{align}
	 \label{eq:zerg1}
	\Pr{N_n = \ell} \le \exp(-Cn)
\end{align}
whenever $\ell \ge Bn$, and
\begin{align}
	\label{eq:zerg2}
	\Pr{N_n = \ell} \le  O(n^{-\alpha}) \exp\left(- C \ell^{\frac{1}{1-\alpha}} n^{- \frac{\alpha}{1- \alpha}}\right)
\end{align}
whenever $An^\alpha \le \ell \le Bn$. Indeed, the factor $\frac{v_\ell \rho_v^\ell}{ u_n \rho_w^n}$ multiplied with $\Pr{S_n = \ell}$ in  Equation~\eqref{eq:extref} (for $\beta= 1/a_1$) evaluates to $O(n^{1+\alpha/a_1} / \ell^{1 + 1/a_1} )$ in this setting. For Inequality~\eqref{eq:zerg1}, the exponential factor in Item 1 of Proposition~\ref{prop:tight445} ``swallows'' this rational function in $\ell$ and $n$, because we may always replace the constant $C$ by a smaller constant. For Inequality~\eqref{eq:zerg2}, we have to combine the factor with the one in Item 2 of Proposition~\ref{prop:tight445} before the exponential term there, which evaluates to
\begin{align*}
	O(n^{1+\alpha/a_1} / \ell^{1 + 1/a_1} )\sqrt{\ell^{\frac{1}{1-\alpha}} n^{-\frac{2 - \alpha}{1-\alpha} }} &= O(1) \left( n^{2\alpha(1-\alpha) - \alpha a_1} \ell^{ (2\alpha-1)a_1 - 2(1-\alpha)} \right)^{\frac{1}{2(1-\alpha)a_1}} \\&= O(n^{-\alpha})
\end{align*}
since $\ell \ge A n^{\alpha}$. Inequalities~\eqref{eq:zerg1} and~\eqref{eq:zerg2} allow us simplify Inequality~\eqref{eq:blockfabrik} to
\begin{align*}
	\Pr{\cE_1} \le \quad  & O(1)  \exp( -c_1 (x / n^{1- 1 / a_d})^{\delta} ) \\+ &O(n^{-\alpha}) \sum_{ \mathclap{n^{\alpha + \epsilon} \le \ell \le Bn }} \exp\left( - \left(C \ell^{\frac{1}{1-\alpha}} n^{- \frac{\alpha}{1- \alpha}} + c_1 (x / \ell^{1- 1 / a_1})^{\delta_1} \right)\right). 
\end{align*}
Now, if the sum index $\ell$ and the parameter $x$ are so that $ \ell^{\frac{1}{1-\alpha}} n^{- \frac{\alpha}{1- \alpha}} < (x / n^{1- 1 / a_d})^{\delta}$, then by Inequality~\eqref{eq:sanity}, Inequality~\eqref{eq:ineqas} and $\delta_1 > \delta$
\begin{align*}
\frac{n^{(1 - 1/a_d)\delta}}{\ell^{(1- 1/a_1)\delta_1}} &\ge \frac{n^{(1 - 1/a_d)\delta}}{ n^{\alpha (1- 1/a_1) \delta_1} (x / n^{1- 1 / a_d})^{\delta(1-\alpha)(1- 1/a_1) \delta_1} } \\
&\ge \frac{n^{(1 - 1/a_d)\delta}}{ n^{\alpha (1- 1/a_1) \delta_1}} \\
&\ge n^{(1 - 1/a_d)(\delta - \delta_1)} \\
&\ge x^{\delta - \delta_1}.
\end{align*}
Hence, there are constants $C_2,c_2>0$ (independent of $\ell$ and $n$) with
\begin{align}
	\label{eq:stereol}
	\Pr{\cE_1} \le C_2  \exp( -c_2 (x / n^{1- 1 / a_d})^{\delta} ).
\end{align}

We now show Inequality~\eqref{eq:mulbound} by induction on $d$. As mentioned above, the case $d = 1$ has already been treated in the literature. We already have a bound for $\Pr{\cE_1}$, so tend to the event $\cE_2$.  We let $K_{(1)} \ge \ldots K_{(N_n)} \ge 1$ denote the number of vertices in the $N_n$ supertrees attached to the first level. Thus
\begin{multline*}
	\Pr{\cE_{2}}  =  \sum_{\ell=1}^{\lfloor n/2 \rfloor} \Pr{N_n = \ell} \sum_{ \substack{ k_1 \ge \ldots \ge k_\ell \ge 1 \\ k_1 + \ldots + k_\ell = n - \ell}} \Pr{K_{(1)} = k_1, \ldots, K_{(\ell)} = k_\ell } p_{k_1, \ldots, k_\ell},
\end{multline*}
with $p_{k_1, \ldots, k_\ell}$ denoting the probability that, conditioned on the number of vertices, one of the supertrees attached to the first level has height at least $x/2$. Setting  $\|\bm{k}\|_1 = \sum_{i=1}^\ell k_i$ (which is always equal to $n-\ell$ in the sum above), and $r = (1-1/a_d)\delta$, it follows from the induction hypotheses that there exist constants $C_3,c_3>0$ with
\begin{align*}
	p_{k_1, \ldots, k_\ell} 
	&\le C_3 \sum_{i=1}^\ell \exp\left(-c_3 \left(x / k_i^{1 - 1/a_d} \right)^\delta\right) \\
	&= C_3 \sum_{i=1}^\ell \exp\left(-c_3 \left(x / \|\bm{k}\|_1^{1 - 1/a_d} \right)^\delta\right)^{1 /q_i^r },
\end{align*}
with $q_i = k_i / \|\bm{k}\|_1$ so that $q_1 + \ldots + q_\ell= 1$. We hence seek to maximize a function \[
f_\ell(x_1, \ldots, x_\ell) = \sum_{i=1}^\ell A^{1/x_i^r}
\]
for $A =  \exp\left(-c_3 \left(x / \|\bm{k}\|_1^{1 - 1/a_d} \right)^\delta\right) <1$ on the set 
\[
M_\ell = \{ (x_i)_{1 \le i \le \ell} \in \ndR_{>0}^\ell  \mid \sum_{i=1}^\ell x_i = 1\}.
\]

For now we consider the case where $x$ is large enough so that
\begin{align}
	\label{eq:bandaid2}
	c_3 x^2/\|\bm{k}\|_1 >  \frac{(1 + r)2^{r}}{r}.
\end{align}

Clearly the function $f_\ell$ is continuous on $\ndR_{\ge 0}^\ell$  and hence assumes a maximum on the closed hull $\overline{M}_\ell$. Let $\bm{y} \in \overline{M}_\ell$ be any point where this maximum is assumed. In case $\bm{y} \in M_\ell$ by Lagrange multipliers it follows that $\bm{y} = (1/\ell, \ldots, 1 / \ell)$. This is because Inequality~\eqref{eq:bandaid2} implies that
\[
	\frac{\mathrm{d}^2}{\mathrm{d}q^2} A^{1 / q^{r}} = \frac{A^{1 / q^r}r \log(A) ((1+r)q^r + r \log(A))}{q^{2(1+ r)}} > 0
\]
for all $0<q<1$.
In case $\bm{y} \notin M_\ell$,  the vector $\tilde{\bm{y}}$, obtained from $\bm{y}$ by deleting all coordinates equal to zero, lies in $M_{\tilde{\ell}}$ for some integer $1 \le \tilde{\ell} < \ell$. Since $\sup_{ \tilde{\bm{x}} \in \overline{M}_{\tilde{\ell}}} f_{\tilde{\ell}}(\tilde{\bm{x}}) \le \sup_{\bm{x} \in \overline{M}_\ell } f_{\ell}(\bm{x})$, it holds that $f_{\tilde{\ell}}$ assumes its maximum on $\overline{M}_{\tilde{\ell}}$ at the point $\tilde{\bm{y}} \in M_{\tilde{\ell}}$, and clearly $f_{\tilde{\ell}}(\tilde{\bm{y}}) = f_{\ell}(\bm{y})$. By Lagrange multipliers it then follows that $\tilde{\bm{y}} = (1/\tilde{\ell}, \ldots, 1/\tilde{\ell})$.

Thus
\begin{align}
	\label{eq:tightstable}
	p_{k_1, \ldots, k_\ell} &\le C_3  \sup_{1 \le i \le \ell} i  \exp \left(-c_3 i^{r} \left(x / \|\bm{k}\|_1^{1 - 1/a_d} \right)^\delta \right).
\end{align}
Using Inequality~\eqref{eq:bandaid2}, it follows that there exists a constants $c_4>0$  that  does not depend on $\ell$ or $k_1, \ldots, k_\ell$ so that
\begin{align*}
	p_{k_1, \ldots, k_\ell} &\le C_3\exp\left(-c_4 \left( x / \|\bm{k}\|_1^{1 - 1/a_d} \right)^\delta \right)
\end{align*}
uniformly for all $x>0$ satisfying Inequality~\eqref{eq:bandaid2}. For all $x<0$ that do not satisfy  Inequality~\eqref{eq:bandaid2} it holds that $\exp\left(-c_4 \left( x / \|\bm{k}\|_1^{1 - 1/a_d} \right)^\delta \right) > \exp\left(- \frac{c_4}{c_3} \frac{(1 + r)2^{r}}{r}\right)$. Hence, setting $C_4 = \max(C_3,  \exp( \frac{c_4}{c_3} \frac{(1 + r)2^{r}}{r}))$ it follows using $	p_{k_1, \ldots, k_\ell} \le 1$ that
\begin{align}
	\label{eq:realtight2}
	p_{k_1, \ldots, k_\ell} &\le C_4   \exp\left(-c_4 \left( x / \|\bm{k}\|_1^{1 - 1/a_d} \right)^\delta \right)
\end{align}
uniformly for all $x>0$.
Hence,
\[
\Pr{\cE_{2}}  \le C_4   \exp\left(-c_4 \left( x / n^{1 - 1/a_d} \right)^\delta \right).
\]
Combining this with Inequality~\eqref{eq:stereol} it follows that there exist constants $C_5, c_5>0$ that do not depend on $n$ or $x$ with
\[
\Pr{ \He(\mT_n) \ge x } \le \Pr{\cE_1} + \Pr{\cE_2} \le C_5\exp\left(-c_5 \left( x / n^{1 - 1/a_d} \right)^\delta \right).
\]
This completes the proof of Inequality~\eqref{eq:mulbound}.

With Inequality~\eqref{eq:realtight2}, the  analogon of the tightness condition~\eqref{eq:step1} (with $\sqrt{n}$ replaced by $n^{1-1/a_d}$) may now be verified analogously as in the proof of Theorem~\ref{te:main}. Hence, the proof of Theorem~\ref{te:mulscal} is now complete.
\end{proof}

\subsection{The dense regime}

Consider a class of supertrees so that the associated composition scheme $V(W(x))$ satisfies the conditions of~\cite[Thm. 3.1]{gibbsmcta}. Then a random $n$-sized supertree the first level has a size that is linear in $n$, and the tree components in the second level have a bounded average size, with the largest ones scaling at the order $n^{1/\alpha}$ for some $1 < \alpha < 2$, or possibly even at most at logarithmic order. Usually that means the first level dominates the global shape, but we also outline an example where the second level does.

\subsubsection{P\'olya trees}

P\'olya trees are rooted unordered unlabelled trees. Their ordinary generating series $A(z)$ satisfies the well-known equation
\[
	A(z) = z \exp\left( \sum_{i \ge 1} A(z^i)/i \right),
\]
since any P\'olya tree corresponds to a root vertex with a multiset of other P\'olya trees attached to it. Otter~\cite{MR0025715} derived the asymptotic  formula for the number $a_n$ of $n$-vertex P\'olya trees
\begin{align}
	\label{eq:polya}
	a_n \sim c_A n^{-3/2} \rho^{-n}
\end{align}
for constants $c_A \approx 0.439924$ and $\rho \approx 0.338321$, and showed that $A(\rho) = 1$.

Letting $T(z) = \sum_{n \ge 1} \frac{n^{n-1}}{n!} z^n$ denote the generating series for rooted unordered labelled trees, the equation
\[
	T(z) = z \exp(T(z))
\]
entails by functional inversion that
\[
	A(z) = T\left( z\exp\left( \sum_{i \ge 2} A(z^i)/i \right) \right). 
\]
This equation in fact corresponds to a structural result that identifies P\'olya trees as a special case of supertrees: There is a $1$ to $n!$ correspondence between $n$-vertex P\'olya trees and pairs $(t, \sigma)$ of an $n$-vertex labelled rooted tree $t$ and an automorphism $\sigma$ of~$t$. These pairs are called symmetries. The fixed points of $\sigma$ form a subtree $t^f$ of $T$. Each vertex of $t^f$ is attached to further branches that $\sigma$ permutes in a fixed-point free manner. Thus, each vertex of $t^f$ is attached to its own (possibly empty) fixed-point free symmetry. Hence the composition of the generating series $T(z)$ for the fixed point tree with the product of $z$ and the generating series $\exp\left( \sum_{i \ge 2} A(z^i)/i \right)$ for fixed-point free symmetries.

This composition scheme satisfies the conditions of~\cite[Thm. 3.1, Lem. 3.5]{gibbsmcta}, so that the size of the first level concentrates in a $\sqrt{n}$ window around $n/\Ex{X}$ for the random non-negative integer $X$ with probability generating series
\[
	\Ex{z^X} = \rho z \exp\left( 1 + \sum_{i \ge 2} A((\rho z)^i)/i \right),
\]
and the maximal size of a level $2$ tree component is at most logarithmic in $n$. By Aldous scaling limit~\cite{MR1085326,MR1207226}, a uniform $n$-sized labelled tree admits the Brownian tree $(\cT_{\mathrm{e}},d_{\cT_{\mathrm{e}}}, \mu_{\cT_{\mathrm{e}}})$ as scaling limit after multiplying distances by $1 / (2 \sqrt{n})$. Consequently, with $\mA_n$ denoting the uniform $n$-vertex P\'olya tree equipped with the uniform probability measure $\mu_{\mA_n}$on its vertex set
\[
	(\mA_n, \sqrt{\Ex{X} / (4n)} d_{\mA_n}, \mu_{\mA_n}) \convdis (\cT_{\mathrm{e}},d_{\cT_{\mathrm{e}}}, \mu_{\cT_{\mathrm{e}}})
\]
in the Gromov--Hausdorff--Prokhorov topology. This result was proven first by~\cite{MR3050512} using Markov-branching trees, and then recovered by~\cite{MR3773800} using Bienaym\'e--Galton--Watson branching processes. To be precise, the size of the first level and the maximal size of the tree-components of the second level in the present supertrees viewpoint imply Gromov--Hausdorff convergence, and the extension to the Gromov--Hausdorff--Prokhorov is by either arguing analogously to~\cite{MR3729639} where it was shown that exchangeable components cause a linear scaling of the mass measure, or alternatively by branching process methods~\cite{MR3773800} which yield convergence of the contour and height processes.

\subsubsection{Bundles of loops}
\label{sec:bundle}

We outline an example where the second level components to dominate.
For $i=1,2$  consider a random non-negative integer $\xi_i$ with $0<\Ex{\xi_i}<1$, $\Pr{\xi_i=0}>0$, and
\[
	\Pr{\xi_i = n} \sim c_1 n^{-1 -\alpha_i}
\]
for some $c_i>0$ and $1 < \alpha_i < 2$.
The looptree of a tree may be obtained by blowing up each vertex into a loop whose circumference is equal to the degree of the vertex~\cite{MR3286462}. Consider the structure $\mL$ obtained taking a $\xi_1$-Bienaym\'e--Galton--Watson tree and attaching to each of its vertices an independent copy of the looptree corresponding to a $\xi_2$-Bienaym\'e--Galton--Watson tree. We assume that
\[
	\alpha_2 > \alpha_1.
\]
and let $\mL_n$ denote the result of conditioning $\mL$ on having $n$ vertices in total.

The size of a $\xi_2$-Bienaym\'e--Galton--Watson tree has first moment
\[
	\mu = \frac{1}{1 -\Ex{\xi_2}}.
\]
The assumptions of~\cite[Thm. 3.1]{gibbsmcta} are satisfied, so that the first level of this superstructure has a size $N_n = n/\mu + O_p(n^{1/{\alpha_2}})$ (in fact the fluctuations  approach an ${\alpha_2}$-stable law). Let $\mathbb{D}([0,1], \ndR)$ denote the set of càdlàg functions on the unit interval, equipped with the Skorokhod $J_1$-topology. Let $(Y_s)_{s \ge 0}$ denote the spectrally positive L\'evy process with Laplace exponent
\begin{align}
	\Exb{e^{-t Y_s }} = \exp(s t^{\alpha_2}).
\end{align}
By~\cite[Thm. 3.3]{gibbsmcta} the component sizes $K_1, \ldots, K_{N_n}$ satisfy
\begin{align*}
		\left( \frac{\sum_{i=1}^{\lfloor s N_n \rfloor}K_i - N_n s \mu }{C n^{1/{\alpha_2}}}, \,\, 0 \le s \le 1 \right) \convdis (\mu Y_s, \,\, 0 \le s \le 1)
\end{align*}
in $\mathbb{D}([0,1], \ndR)$ as $n \to \infty$, with
\[
	C = \mu^{-1-1/\alpha_2} \left( \frac{c_2}{(1 - \Ex{\xi_2})^{1+\alpha_2}} \frac{\Gamma(1- \alpha_2)}{\alpha_2}\right)^{1/\alpha_2}.
\] Thus, the jumps  of $(\mu Y_s, \,\, 0 \le s \le 1)$ correspond to the maximal level 2 component sizes (at order $n^{1/\alpha_2}$).

By~\cite{MR3335012} (see also~\cite{MR4144886}) the looptree of a $\xi_2$-tree conditioned to have $m$ vertices looks like a linearly sized loop of diameter $m/\mu + O_p(m^{1/\alpha_2})$ with $O_p(m^{1/\alpha_2})$-sized components attached to it. The first level on the other hand, again by results of~\cite{MR3335012}, has logarithmic diameter. 

Hence, if we multiply all distances in $\mL_n$  by $1 / (C n^{1/\alpha_2})$, then the first level contracts to a single point. The mass of the level-$2$ components is concentrated on the small components by~\cite{gibbsmcta}, thus the entire mass of $\mL_n$ concentrates on this single point. For any fixed $k \ge 1$ the macroscopic loops in the $k$ largest level-$2$ components converge jointly after rescaling  by $1 / (C n^{1/\alpha_2})$  to the $k$ largest jump heights of $(Y_s, 0 \le s \le 1)$.

Thus, $\mL_n$  equipped with the uniform probability measure $\mu_{\mL_n}$ on its set of vertices satisfies
\[
	(\mL_n, \frac{1}{C n^{1/\alpha_2}} d_{\mL_n}, \mu_{\mL_n}) \convdis \mathcal{B}(\alpha)
\]
in the Gromov--Hausdorff--Prokhorov sense as $n \to \infty$,  with the limit $\mathcal{B}(\alpha)$  consisting of a single point  carrying all the mass to which we attach countably infinite many massless loops whose circumferences equal  the jump heights of  the process $(Y_s, 0 \le s \le 1)$. The tightness condition required in deducing this convergence boils down to showing that $k$th largest jump of $(Y_s, 0 \le s \le 1)$ converges in probability to zero as $k \to \infty$, and this clear from the explicit description of its density function~\cite[Cor. 3.4]{gibbsmcta}.

\subsection{The convergent regime}

Supertrees whose  associated composition scheme $V(W(x))$ falls into the convergent regime of~\cite[Thm. 3.11]{gibbsmcta} have the property, that the first level remains stochastically bounded, as do the tree components in the second level, except for a single giant level $2$ tree component. Thus the scaling limit of the supertree is identical to the scaling limit of large trees of the inner class.

\subsubsection{Planar supertrees}
	To illustrate the regime we discuss an  example of $2$-dimensional supertrees, which was mentioned in~\cite[Ex. VI.10]{MR2483235}. 
	
	Let $G$ denote the class of plane trees, such that
	\[
	G(z) = z / (1- G(z)).
	\]
	This way, $G$ has radius of convergence $1/4$, and $G(1/4) = 1/2$.
    Consider the class $K$ of supertrees given by the composition scheme
	\[
	K(z) = G(zG(z)).
	\]
	Thus, a tree from $K$ consists of a first level plane tree from $G$, where each vertex has a child that is the root of some plane tree of the second level. 	The composition is subcritical since $\frac{1}{4} G(\frac{1}{4}) = \frac{1}{8} < \frac{1}{4}$. Hence the conditions of~\cite[Thm. 3.11]{gibbsmcta} are met. This means there is a single giant level $2$ component, and hence the uniform $n$-vertex supertree $\mK_n$ from the class $K$ equipped with the uniform measure $\mu_{\mK_n}$ 
	admits the Brownian tree $(\cT_{\mathrm{e}},d_{\cT_{\mathrm{e}}}, \mu_{\cT_{\mathrm{e}}})$ as limit
\[
	(\mK_n, (2n)^{-1/2} d_{\mK_n}, \mu_{\mK_n}) \convdis 	(\cT_{\mathrm{e}},d_{\cT_{\mathrm{e}}}, \mu_{\cT_{\mathrm{e}}}).
\]

	It's interesting to note that if we make a very slight adjustment to our model by declaring each edge connecting a vertex of the first level with the vertex of the second level to be either a left or right edge, then we look at a class $\tilde{K}$ in the dilute regime with generating series
	\[
		\tilde{K}(z) = G(2zG(z)).
	\]
	The scaling limit of the random $n$-vertex tree $\tilde{\mK}_n$  from the class  $\tilde{K}$ equipped with the uniform probability measure  $\mu_{\tilde{\mK}_n}$ is then no longer the Brownian tree, but instead
		\[
	(\tilde{\mK}_n, (2n)^{-1/2} d_{\tilde{\mK}_n}, \mu_{\tilde{\mK}_n})  \convdis  \cS(1/2, -1/4, 1/2, L_{\mathrm{Brownian}}),
	\]
	with $\mu_{\tilde{\mK}_n}$ denoting the uniform probability measure on the vertex set of $\tilde{\mK}_n$.
	

\subsection{The mixture regime}

In the mixture regime, we obtain a limiting behaviour that is a mixture of the convergent and dense regime. We enter this regime when the  composition scheme $V(W(x))$ associated to the supertrees satisfies the requirements of~\cite[Thm. 3.12]{gibbsmcta}.

For example, if in the random graph $\mL_n$ of Section~\ref{sec:bundle} we would have $\alpha_1 = \alpha_2$ (instead of $\alpha_2 > \alpha_1$), then by~\cite[Thm. 3.12]{gibbsmcta} there is a limiting probability $0<p<1$ such that with asymptotic probability $p$ the random graph $\mL_n$ looks like an  $(n-O_p(1))$-sized level $2$-component with a stochastically bounded rest attached to it, so that the overall shape is that of a circle with circumference $n/\mu$. And on the complementary event, with limiting probability $1-p$ the behaviour is as in Section~\ref{sec:bundle} with a diameter of order $O_p(n^{1/\alpha_2})$. Thus, the scaling limit of $\mL_n$ after rescaling distances by $n^{-1}$ is a mixture of a circle with circumference $1/\mu$ (with probability $p$) and a single point (with probability $1-p$).

\section{Local convergence of Kemp's $d$-dimensional binary trees}

\label{sec:prooflocal}

Let $d \ge 1$. Since $F_d(1/4) = 1$, we may define a  random finite tree $\mF^d$  from the class $\cF_d$ that gets drawn with probability
\[
	\Pr{\mF^d = F} = (1/4)^{|F|}.
\]
Here $|F|$ denotes the number of vertices of the tree $F$. 

In the following, we now suppose that $d \ge 2$. Decomposing the trees from the class $\cF_d$ at the root yields
\[
	F_{d}(x) = x F_{d-1}(x)(1 + 2 F_d(x) + F_d(x)^2).
\]
From this and Equation~\eqref{eq:fdasymp} the following result is elementary (or a consequence of limits for general convergent Gibbs partitions~\cite{MR3854044}):
\begin{lemma}
	\label{le:root}
	 The $\cF_{d-1}$-tree attached to the root of $\mF_{d,n}$ converges to an independent copy of $\mF^{d-1}$.  Jointly, there are four equally likely cases: 	
	\begin{enumerate}
		\item The root vertex is linked via a left edge to a single macroscopic tree from $\cF_d$, and to nothing else.
		\item The root vertex is linked via a right edge to a single macroscopic tree from $\cF_d$, and to nothing else.
		\item The root vertex is linked via a left edge to a single macroscopic tree from $\cF_d$, and via a right edge to an independent copy of $\mF^{d}$.
		\item The root vertex is linked via a right edge to a single macroscopic tree from $\cF_d$, and via a left edge to an independent copy of $\mF^{d}$.
	\end{enumerate}
\end{lemma}

Lemma~\ref{le:root} may be applied again to the unique macroscopic $\cF_d$ component, which then has a unique macroscopic $\cF_d$-component by itself, to which we again may apply Lemma~\ref{le:root}, and so on. Hence
\[
	\mF_{d,n} \convdis \hat{\mF}_d
\]
in the local topology, with  $\hat{\mF}_d$ denote denoting the tree constructed from an infinite path $v_0, v_1, v_2, \ldots$ such that for each $i \ge 0$ the vertex $v_i$ is linked via an edge to an independent copy of $\mF^{d-1}$, and according to a fair independent choice one of the following four cases hold:
\begin{enumerate}
	\item The edge joining $v_i$ and $v_{i+1}$ is declared a left edge.
	\item The edge joining $v_i$ and $v_{i+1}$ is declared a right edge.
	\item The edge joining $v_i$ and $v_{i+1}$ is declared a left edge, and additionally we add a right edge from $v_i$ to the root of an independent copy of $\mF^d$.
	\item The edge joining $v_i$ and $v_{i+1}$ is declared a right edge, and additionally we add a left edge from $v_i$ to the root of an independent copy of $\mF^d$.
\end{enumerate}
This concludes the proof of the first limit in Theorem~\ref{te:main2}. As for the local limit near a random root, Equation~\eqref{eq:kin} and induction on $d$ entail that the extremal sizes $K^d_{(1)} \ge K^d_{(2)} \ge \ldots$ of the level $d$ trees in $\mF_{d,n}$ satisfy
\begin{align}
	\label{eq:witness}
	(K_{(i)}^d/n)_{i \ge 1} \convdis (Y_i)_{i \ge 1}
\end{align}
with $(Y_i)_{i \ge 1}$ denoting the ranked points of the two-parameter Poisson--Dirichlet distribution $\mathrm{PD}(1/2, -1/2^d)$.

We know from~\cite{MR1102319,MR2908619} that there exists a random tree $\bar{\mF}$ such that
\[
(\mF_{1,n},v_n) \convdis \bar{\mF}
\]
in the local topology for re-rooted trees. If we take two independent uniformly selected vertices $v_{n}^{(1)}$ and $v_{n}^{(2)}$ of $\mF_{1,n}$ then their neighbourhoods converge jointly to the neighbourhoods of two independent copies of $\bar{\mF}$. This is equivalent to the quenched convergence
\[
\mathfrak{L}((\mF_{1,n},v_n) \mid \mF_{1,n}) \convp \mathfrak{L}(\bar{\mF}).
\]

Since $\sum_{i \ge 1} Y_i = 1$, it follows from Equation~\eqref{eq:witness} that a uniformly selected vertex of $\mF_{d,n}$ falls with high probability into one of the level $d$ tree components. Since neighbourhoods of the root of large trees in $\cF_1$ have a stochastically bounded size (see the local convergence in~\cite{MR2908619}), it follows that a random vertex in a large tree from $\cF_1$ has a height that converges in probability towards infinity. Thus, (fixed radius) neighbourhoods of the random vertex $v_n$ of $\mF_{d,n}$ are with high probability entirely contained in the level $d$ tree component that $v_n$ falls into. Thus
\[
	(\mF_{d,n},v_n) \convdis \bar{\mF}.
\]
If we take two independent uniform vertices 
of $\mF_{d,n}$, then there asymptotically two cases: Either they fall into different level $d$ components, in which their neighbourhoods converge in distribution towards neighbourhoods of independent copies of $\bar{\mF}$. Or they fall into the same level $d$ component, and  by the quenched convergence of $\mF_{1,n}$ their neighbourhoods then also converge towards neighbourhoods of two independent copies of $\bar{\mF}$. This proves quenched convergence
\[
\mathfrak{L}((\mF_{d,n},v_n) \mid \mF_{d,n}) \convp \mathfrak{L}(\bar{\mF}).
\]
Hence the proof of Theorem~\ref{te:main2} is complete.

\section*{Acknowledgement}

I warmly thank the referee for the thorough reading and helpful remarks. I am grateful to Markus Heydenreich and Andrei Shubin for helpful discussions.

\section*{Declarations}

\emph{Conflict of interest statement:} We declare that the author has no competing interests as defined by Springer, or other interests that might be perceived to influence the results and/or discussion reported in this paper.

\emph{Data availability statement:} We do not analyse or generate any datasets, because our work proceeds within a theoretical and mathematical approach.

\bibliographystyle{abbrv}
\bibliography{supertrees}

\begin{thebibliography}{10}

\bibitem{MR3035742}
R.~Abraham, J.-F. Delmas, and P.~Hoscheit.
\newblock {A note on the {G}romov-{H}ausdorff-{P}rokhorov distance between
  (locally) compact metric measure spaces}.
\newblock {\em Electron. J. Probab.}, 18:no. 14, 21, 2013.

\bibitem{MR3077536}
L.~Addario-Berry, L.~Devroye, and S.~Janson.
\newblock Sub-{G}aussian tail bounds for the width and height of conditioned
  {G}alton-{W}atson trees.
\newblock {\em Ann. Probab.}, 41(2):1072--1087, 2013.

\bibitem{addarioberry2022random}
L.~Addario-Berry and S.~Donderwinkel.
\newblock Random trees have height ${O}(\sqrt{n})$.
\newblock {\em arXiv:2201.11773}.

\bibitem{MR3729639}
L.~Addario-Berry and Y.~Wen.
\newblock Joint convergence of random quadrangulations and their cores.
\newblock {\em Ann. Inst. Henri Poincar\'e Probab. Stat.}, 53(4):1890--1920,
  2017.

\bibitem{MR1102319}
D.~Aldous.
\newblock Asymptotic fringe distributions for general families of random trees.
\newblock {\em Ann. Appl. Probab.}, 1(2):228--266, 1991.

\bibitem{MR1085326}
D.~Aldous.
\newblock The continuum random tree. {I}.
\newblock {\em Ann. Probab.}, 19(1):1--28, 1991.

\bibitem{MR1166406}
D.~Aldous.
\newblock The continuum random tree. {II}. {A}n overview.
\newblock In {\em Stochastic analysis ({D}urham, 1990)}, volume 167 of {\em
  London Math. Soc. Lecture Note Ser.}, pages 23--70. Cambridge Univ. Press,
  Cambridge, 1991.

\bibitem{MR1207226}
D.~Aldous.
\newblock The continuum random tree. {III}.
\newblock {\em Ann. Probab.}, 21(1):248--289, 1993.

\bibitem{MR1808372}
D.~Aldous and J.~Pitman.
\newblock Inhomogeneous continuum random trees and the entrance boundary of the
  additive coalescent.
\newblock {\em Probab. Theory Related Fields}, 118(4):455--482, 2000.

\bibitem{zbMATH06610054}
S.~{Athreya}, W.~{L\"ohr}, and A.~{Winter}.
\newblock {The gap between Gromov-Vague and Gromov-Hausdorff-vague topology}.
\newblock {\em {Stochastic Processes Appl.}}, 126(9):2527--2553, 2016.

\bibitem{banderier2021analyticv2}
C.~Banderier, M.~Kuba, and M.~Wallner.
\newblock Phase transitions of composition schemes: Mittag-leffler and mixed
  poisson distributions.
\newblock {\em arXiv:2103.03751}, 2021.

\bibitem{2021arXiv210303751B}
C.~{Banderier}, M.~{Kuba}, and M.~{Wallner}.
\newblock {Phase transitions of composition schemes: Mittag-Leffler and mixed
  Poisson distributions}.
\newblock {\em arXiv e-prints}, page arXiv:2103.03751, Mar. 2021.

\bibitem{MR4038060}
M.~Bloznelis.
\newblock Local probabilities of randomly stopped sums of power-law lattice
  random variables.
\newblock {\em Lith. Math. J.}, 59(4):437--468, 2019.

\bibitem{MR1835418}
D.~Burago, Y.~Burago, and S.~Ivanov.
\newblock {\em A course in metric geometry}, volume~33 of {\em Graduate Studies
  in Mathematics}.
\newblock American Mathematical Society, Providence, RI, 2001.

\bibitem{MR3286462}
N.~Curien and I.~Kortchemski.
\newblock Random stable looptrees.
\newblock {\em Electron. J. Probab.}, 19:no. 108, 35, 2014.

\bibitem{MR1440141}
R.~A. Doney.
\newblock One-sided local large deviation and renewal theorems in the case of
  infinite mean.
\newblock {\em Probab. Theory Related Fields}, 107(4):451--465, 1997.

\bibitem{MR2484382}
M.~Drmota.
\newblock {\em Random trees}.
\newblock SpringerWienNewYork, Vienna, 2009.
\newblock An interplay between combinatorics and probability.

\bibitem{MR1964956}
T.~Duquesne.
\newblock A limit theorem for the contour process of conditioned
  {G}alton-{W}atson trees.
\newblock {\em Ann. Probab.}, 31(2):996--1027, 2003.

\bibitem{MR1954248}
T.~Duquesne and J.-F. Le~Gall.
\newblock Random trees, {L}\'{e}vy processes and spatial branching processes.
\newblock {\em Ast\'{e}risque}, (281):vi+147, 2002.

\bibitem{MR3634265}
T.~Duquesne and M.~Wang.
\newblock Decomposition of {L}\'{e}vy trees along their diameter.
\newblock {\em Ann. Inst. Henri Poincar\'{e} Probab. Stat.}, 53(2):539--593,
  2017.

\bibitem{MR2483235}
P.~Flajolet and R.~Sedgewick.
\newblock {\em Analytic combinatorics}.
\newblock Cambridge University Press, Cambridge, 2009.

\bibitem{MR2193117}
B.~Gittenberger and A.~Panholzer.
\newblock Some results for monotonically labelled simply generated trees.
\newblock In {\em 2005 {I}nternational {C}onference on {A}nalysis of
  {A}lgorithms}, Discrete Math. Theor. Comput. Sci. Proc., AD, pages 173--180.
  Assoc. Discrete Math. Theor. Comput. Sci., Nancy, 2005.

\bibitem{MR3050512}
B.~Haas and G.~Miermont.
\newblock Scaling limits of {M}arkov branching trees with applications to
  {G}alton-{W}atson and random unordered trees.
\newblock {\em Ann. Probab.}, 40(6):2589--2666, 2012.

\bibitem{MR2597584}
K.~Handa.
\newblock The two-parameter {P}oisson-{D}irichlet point process.
\newblock {\em Bernoulli}, 15(4):1082--1116, 2009.

\bibitem{2011arXiv1112.0220J}
S.~{Janson}.
\newblock {Stable distributions}.
\newblock {\em arXiv e-prints}, page arXiv:1112.0220, Dec. 2011.

\bibitem{MR2908619}
S.~Janson.
\newblock Simply generated trees, conditioned {G}alton-{W}atson trees, random
  allocations and condensation.
\newblock {\em Probab. Surv.}, 9:103--252, 2012.

\bibitem{janson2020gromovprohorov}
S.~Janson.
\newblock On the {G}romov-{P}rohorov distance.
\newblock {\em arXiv:2005.13505}, 2020.

\bibitem{MR818219}
O.~Kallenberg.
\newblock {\em Random measures}.
\newblock Akademie-Verlag, Berlin; Academic Press, Inc. [Harcourt Brace
  Jovanovich, Publishers], London, third edition, 1983.

\bibitem{MR1260502}
R.~Kemp.
\newblock Monotonically labelled ordered trees and multidimensional binary
  trees.
\newblock In {\em Fundamentals of computation theory ({S}zeged, 1993)}, volume
  710 of {\em Lecture Notes in Comput. Sci.}, pages 329--341. Springer, Berlin,
  1993.

\bibitem{zbMATH00168426}
R.~Kemp.
\newblock Random multidimensional binary trees.
\newblock {\em J. Inf. Process. Cybern.}, 29(1):9--36, 1993.

\bibitem{MR1370950}
R.~Kemp.
\newblock On the inner structure of multidimensional simply generated trees.
\newblock In {\em Proceedings of the {S}ixth {I}nternational {S}eminar on
  {R}andom {G}raphs and {P}robabilistic {M}ethods in {C}ombinatorics and
  {C}omputer {S}cience, ``{R}andom {G}raphs '93'' ({P}ozna\'{n}, 1993)},
  volume~6, pages 121--146, 1995.

\bibitem{MR2251473}
D.~E. Knuth.
\newblock {\em The art of computer programming. {V}ol. 4, {F}asc. 4}.
\newblock Addison-Wesley, Upper Saddle River, NJ, 2006.
\newblock Generating all trees---history of combinatorial generation.

\bibitem{MR3335012}
I.~Kortchemski.
\newblock Limit theorems for conditioned non-generic {G}alton-{W}atson trees.
\newblock {\em Ann. Inst. Henri Poincar\'e Probab. Stat.}, 51(2):489--511,
  2015.

\bibitem{MR3651047}
I.~Kortchemski.
\newblock Sub-exponential tail bounds for conditioned stable
  {B}ienaym\'{e}-{G}alton-{W}atson trees.
\newblock {\em Probab. Theory Related Fields}, 168(1-2):1--40, 2017.

\bibitem{MR2571957}
G.~Miermont.
\newblock Tessellations of random maps of arbitrary genus.
\newblock {\em Ann. Sci. \'Ec. Norm. Sup\'er. (4)}, 42(5):725--781, 2009.

\bibitem{MR0025715}
R.~Otter.
\newblock The number of trees.
\newblock {\em Ann. of Math. (2)}, 49:583--599, 1948.

\bibitem{MR3773800}
K.~Panagiotou and B.~Stufler.
\newblock Scaling limits of random {P}\'{o}lya trees.
\newblock {\em Probab. Theory Related Fields}, 170(3-4):801--820, 2018.

\bibitem{MR2245368}
J.~Pitman.
\newblock {\em Combinatorial stochastic processes}, volume 1875 of {\em Lecture
  Notes in Mathematics}.
\newblock Springer-Verlag, Berlin, 2006.
\newblock Lectures from the 32nd Summer School on Probability Theory held in
  Saint-Flour, July 7--24, 2002, With a foreword by Jean Picard.

\bibitem{MR1434129}
J.~Pitman and M.~Yor.
\newblock The two-parameter {P}oisson-{D}irichlet distribution derived from a
  stable subordinator.
\newblock {\em Ann. Probab.}, 25(2):855--900, 1997.

\bibitem{MR3914354}
L.~Ramzews and B.~Stufler.
\newblock Simply generated unrooted plane trees.
\newblock {\em ALEA Lat. Am. J. Probab. Math. Stat.}, 16(1):333--359, 2019.

\bibitem{gibbsmcta}
B.~Stufler.
\newblock Gibbs partitions: a comprehensive phase diagram.
\newblock {\em Ann. Inst. H. Poincaré Probab. Statist., to appear}.

\bibitem{MR3854044}
B.~Stufler.
\newblock Gibbs partitions: the convergent case.
\newblock {\em Random Structures Algorithms}, 53(3):537--558, 2018.

\bibitem{MR3983790}
B.~Stufler.
\newblock The continuum random tree is the scaling limit of unlabeled unrooted
  trees.
\newblock {\em Random Structures Algorithms}, 55(2):496--528, 2019.

\bibitem{MR4144886}
B.~Stufler.
\newblock On the maximal offspring in a subcritical branching process.
\newblock {\em Electron. J. Probab.}, 25:Paper No. 104, 62, 2020.

\bibitem{zbMATH05306371}
C.~{Villani}.
\newblock {\em {Optimal transport. Old and new}}, volume 338.
\newblock Berlin: Springer, 2009.

\bibitem{MR3573447}
M.~Wang.
\newblock Scaling limits for a family of unrooted trees.
\newblock {\em ALEA Lat. Am. J. Probab. Math. Stat.}, 13(2):1039--1067, 2016.

\end{thebibliography}

\end{document}